\documentclass[12pt,reqno]{amsart}

\usepackage{amsmath,amsthm,amscd,amsfonts,amssymb,graphicx,color}

\usepackage[bookmarksnumbered,colorlinks,plainpages]{hyperref}
\hypersetup{colorlinks=true,linkcolor=blue,anchorcolor=green,citecolor=cyan,urlcolor=blue,filecolor=blue,pdftoolbar=true}

\usepackage{mathrsfs}

\usepackage{orcidlink}

\usepackage{bm}

\usepackage{tikz}
\usetikzlibrary{arrows.meta, positioning, shapes}

\newtheorem{theorem}{Theorem}[section]
\newtheorem{lemma}[theorem]{Lemma}
\newtheorem{proposition}[theorem]{Proposition}

\newtheorem{assumption}[theorem]{Assumption}
\theoremstyle{definition}
\newtheorem{definition}[theorem]{Definition}
\theoremstyle{remark}
\newtheorem{remark}[theorem]{Remark}

\numberwithin{equation}{section}

\newcommand{\BR}{{\mathbb R}}
\newcommand{\BC}{{\mathbb C}}
\newcommand{\BN}{{\mathbb N}}

\newcommand{\D}{\bm D}

\newcommand{\X}{\bm x}
\newcommand{\XI}{\bm \xi}
\renewcommand{\P}{{\bm P}}

\newcommand{\Y}{\bm y}

\newcommand{\cl}{C \kern -0.1em \ell}

\newcommand{\e}{{\bm e}}
\renewcommand{\a}{\bm a}
\renewcommand{\b}{\bm b}

\newcommand{\dx}{{\bm d}{\bm x}}
\newcommand{\dy}{{\bm d}{\bm y}}
\newcommand{\dxi}{{\bm d}{\bm \xi}}

\newcommand{\f}{\bm f}
\newcommand{\g}{\bm g}

\renewcommand{\u}{\bm u}

\newcommand{\w}{\bm w}
\newcommand{\PSI}{\bm \psi}

\begin{document}
	
	\setcounter{page}{1}
	
	\title[Gauge Transforms and Strichartz Estimates for BOD Equations]{Gauge Transforms and Strichartz Estimates for Benjamin--Ono-Type Dirac Equations}

	
	\author{{{Nelson Faustino}~\orcidlink{0000-0002-9117-2021}}}
	\address[\textsc{Nelson~Faustino}]{Department of Mathematics (DMAT) and Center for R\&D in Mathematics and Applications (CIDMA), Department of Mathematics, University of Aveiro, Campus Universit\'ario de Santiago, 3810-193 Aveiro, Portugal}
	\email{\href{mailto:nfaust@ua.pt}{nfaust@ua.pt}}
	\thanks{{Nelson~Faustino} is supported by CIDMA (\href{https://ror.org/05pm2mw36}{https://ror.org/05pm2mw36})
		under the Portuguese Foundation for Science and Technology 
		(FCT, \href{https://ror.org/00snfqn58}{https://ror.org/00snfqn58}), Grants 
		UID/04106/2025 (\href{https://doi.org/10.54499/UID/04106/2025}{https://doi.org/10.54499/UID/04106/2025})
		and UID/PRR/04106/2025}


	
\begin{abstract}
For $T>0$, we consider the Benjamin--Ono-type Dirac equation
$$
i\partial_t{\bm\psi}({\bm x},t)
+
{\bm D}(-\Delta)^{\frac{1}{2}}{\bm\psi}({\bm x},t)
=
{\bm D}\,{\bm G}({\bm\psi}({\bm x},t)),
\qquad
({\bm x},t)\in\mathbb{R}^n\times[0,T],
$$
where $(-\Delta)^{\frac{1}{2}}$ is the fractional Riesz operator,
${\bm D}$ is the Dirac operator associated with the Clifford algebra
of signature $(0,n)$, and ${\bm G}(\cdot)$ a scalar-valued quadratic
nonlinearity.
Under a pointwise anti-commutation condition on ${\bm\psi}$ and a
uniform dyadic regularity assumption on its spatial Fourier profile,
we construct a unit-modulus gauge map  $\u$ by
$
\u=\exp\bigl(-i{\bm w}\bigr),
$
where ${\bm w}$ is a solution of
${\bm D}{\bm w}={\bm \psi}$.
Assuming, in addition, the stated regularity and commutation conditions
on ${\bm w}$, we show that the Plemelj components
${\bm u}_{\pm}={\bm P}_{\pm}{\bm u}$
solve a coupled Schr\"odinger system with gauge nonlinearities
${\bm P}_{\pm}{\bm N}({\bm\psi},{\bm u})$, where
$$
{\bm N}({\bm\psi}({\bm x},t),{\bm u}({\bm x},t))
=
-2{\bm P}_{-}\!\left[
{\bm G}({\bm\psi}({\bm x},t))
\right]{\bm u}({\bm x},t).
$$
We also obtain a priori estimates for the homogeneous Sobolev
seminorms $\dot H^\sigma$ of ${\bm\psi}(\cdot,t)$, with at most
exponential growth on $[0,T]$, together with estimates for
${\bm P}_{\pm}{\bm N}({\bm\psi}(\cdot,t),{\bm u}(\cdot,t))$.
Combining these bounds with the homogeneous and inhomogeneous Strichartz
estimates for the Schr\"odinger system, we obtain, for $n\geq 3$ and
every sharp Schr\"odinger admissible pair $(q,p)$, an a priori
Strichartz-type estimate for ${\bm \psi}$ in
$L_T^q\dot W_{\bm x}^{\sigma,p}$, expressed in terms of the initial datum
${\bm\psi}(\cdot,0)$ and the a priori controls of ${\bm\psi}$ in
$L_T^\infty H_{\bm x}^{\sigma+1}$ and $L_T^\infty L_{\bm x}^\infty$.
\end{abstract}

\subjclass[2020]{30G35,~35Q41,~35Q55,~35S30}

\keywords{Benjamin--Ono-type dispersion, Dirac-type equations, gauge transforms, Riesz--Hilbert transform, Strichartz estimates}

	\maketitle
	
	\tableofcontents

\section{Introduction}\label{Introduction}

\subsection{State of the art}

We study a Benjamin--Ono-type Dirac equation in the Clifford algebra setting,
\begin{equation}
	\label{BO-Dirac}
	\tag{$\mathcal{BO}$}
	i\partial_t \PSI(\X,t)+\D(-\Delta)^{\frac{1}{2}}\PSI(\X,t)
	=\D{\bm G}(\PSI(\X,t)), \qquad (\X,t)\in\mathbb{R}^{n}\times [0,T],
\end{equation}
and develop a gauge-based approach that reduces this equation to a coupled system of nonlinear Schr\"odinger equations.

Here $(-\Delta)^{\frac{1}{2}}$ is the fractional Riesz operator and $\D$ is the Dirac operator
\begin{equation}
	\label{DiracOp}
	\D=\sum_{j=1}^n \e_j \partial_{x_j},
\end{equation}
associated with the Clifford algebra of signature $(0,n)$, $\cl_{0,n}$ (see Subsection~\ref{CliffordAlgebras}). For the nonlinear term ${\bm G}(\cdot)$, we take
\begin{equation}\label{Gu}
	\tag{$\mathcal{G}$}
	\begin{aligned}
		{\bm G}(\PSI(\X,t))
		= -i\,\big( \PSI(\X,t)\big)^\dagger\PSI(\X,t),
		~~~(\X,t)\in \BR^n \times [0,T],
	\end{aligned}
\end{equation}
where $\dagger$ denotes the conjugation on the complexified Clifford algebra $\BC\otimes\cl_{0,n}$ (see \eqref{conjugation} in Subsection~\ref{CliffordAlgebras}). 

Since \(\a^\dagger \a\) is scalar for every
\(\a\in\BC\otimes \cl_{0,n}\), we have
\begin{equation}\label{Gu0}
	\begin{aligned}
		{\bm G}(\PSI(\X,t))
		= -i~2^{-n}\big| \PSI(\X,t)\big|^2_0,
		~~~(\X,t)\in \BR^n \times [0,T].
	\end{aligned}
\end{equation}
Here, \(|\cdot|_0\) denotes the Clifford norm in \(\BC\otimes\cl_{0,n}\) (see \eqref{InnerProduct-0} in Subsection \ref{CliffordFunctionSpaces}).

The symbol of $\D(-\Delta)^{\frac{1}{2}}$ is $i\XI|\XI|$, so the dispersion is of Benjamin--Ono-type. It is first order in $|\XI|$, but still carries the Dirac structure through the Clifford vector $i\XI$. In this sense, \eqref{BO-Dirac} is a space-fractional (and hence non-local)
analogue of the nonlinear Dirac equations studied by Escobedo and Vega~\cite{EVega97} and later developed, in the small-data critical-regularity regime, by Machihara, Nakanishi, and Ozawa~\cite{MNOzawa03,MNOzawa05}, Bejenaru and Herr~\cite{BeHerr15,BeHerr16}, and Bournaveas and Candy~\cite{BouCandy15}.

It is worth noting that the spectral identity $\D(-\Delta)^{\frac{1}{2}}=-\Delta \mathcal{H}$, written in terms of the Laplacian $\Delta$ and the Riesz--Hilbert transform
\begin{equation}
	\label{RieszHilbertOp}
	\mathcal{H}=\D(-\Delta)^{-\frac{1}{2}},
\end{equation}

This formulation appears in McIntosh~\cite{mcintosh1996clifford} (see also Bernstein~\cite{Bernstein2016} and Dang--Mai--Qian~\cite{DangMaiQian20}).

The scalar Clifford nonlinearity \eqref{Gu} is a quadratic self-interaction
of the Dirac field, analogous to the scalar self-interactions appearing in
nonlinear Dirac models of Soler type~\cite{Soler70} and to the classical
bilinear spinor interaction of Thirring~\cite{Thirring58}. In this sense,
\eqref{Gu} may be regarded as a Clifford-valued extension of these
scalar self-interactions. 

The equation \eqref{BO-Dirac} may be regarded as a Benjamin--Ono-type model with nonlinear Dirac coupling. In their studies of long internal waves in stratified fluids, Benjamin and Ono obtained the scalar Benjamin--Ono equation~\cite{Benjamin67,Ono75}. Later, Pelinovsky and Shrira obtained a two-dimensional model of boundary-layer shear flows~\cite{PShrira95}. Hickman, Linares, Riano, Rogers and Wright~\cite{LinaresRianoRogersWright19} and Riano~\cite{Riano20} independently constructed higher dimensional extensions of this model. Equations of this type occur in Hamiltonian long-wave expansions (see \cite{CGKalish05}).

\subsection{Main results}

The main step in the present approach is the construction of a gauge
transform, closely related to the Cole--Hopf transform for the inviscid
Burgers equation (cf.~\cite{Hopf50,Cole51}). Since the Riesz--Hilbert transform~\eqref{RieszHilbertOp} defines a
unitary operator, Tao's gauge construction for the Benjamin--Ono equation
\cite{Tao04,Tao06} applies in this setting to relate the solution
$\PSI$ of \eqref{BO-Dirac} to the Plemelj components
$\u_+=\P_+\u$ and $\u_-=\P_-\u$ of the gauge map $\u$. These components
are solutions of the nonlinear Schr\"odinger system
\begin{equation}\label{Schr-Alfa}
	\tag{$\mathcal{S}_{\pm}$}
	\begin{cases}
		(i\partial_t-\Delta) \u_+(\X,t)=\P_+{\bm N}\left(\PSI(\X,t),\u(\X,t)\right) \\[0.2em]
		(i\partial_t+\Delta) \u_-(\X,t)=\P_-{\bm N}\left(\PSI(\X,t),\u(\X,t)\right)
	\end{cases},
	(\X,t)\in \mathbb{R}^{n}\times [0,T],
\end{equation}
where $\P_\pm$ are the Plemelj projectors (see~\eqref{ProjOp-RiezHilbert} in
Subsection~\ref{Subsection-FractionalDirac-RieszHilbert}).

The corresponding gauge nonlinearities in \eqref{Schr-Alfa} are given by $\P_\pm{\bm N}(\PSI,\u)$, where
\begin{equation}\label{Nelson_pm}
	\tag{$\mathcal{N}$}
	{\bm N}\bigl(\PSI(\X,t),\u(\X,t)\bigr)
	=
	-2\,\P_-\!\bigl[{\bm G}(\PSI(\X,t))\bigr]\;\u(\X,t),
	\qquad (\X,t)\in \BR^n\times[0,T].
\end{equation}

This construction is related to gauge methods for the Benjamin--Ono
equation developed in \cite{GKT24,Kishimoto22,MolinetPilod12,MP23,Tao04}.
It differs from the integrable-systems construction of Killip, Laurens,
and Vi\c{s}an \cite{KLVisan24}, where the gauge map is obtained from a
modified Lax pair.

In the scalar Benjamin--Ono case, the gauge map has unit modulus. This
is no longer the case for Clifford-valued functions, as Clifford
multiplication may combine components of different grades. Therefore, in general,
\begin{equation}\label{unit-modulus}
	\tag{\(\mathcal{U}\)}
	\u(\X,t)\bigl(\u(\X,t)\bigr)^\dagger=1,
	\qquad
	(\X,t)\in\BR^n\times[0,T],
\end{equation}
need not hold. We impose the following condition on \(\PSI\) in order
to obtain \eqref{unit-modulus}.

\begin{assumption}\label{AntiCommutation-Assumption}
	For every \(j=1,\ldots,n\), the solution \(\PSI\) of
	\eqref{BO-Dirac} satisfies
	\begin{equation}\label{AntiCommutation-Condition}
		\tag{\(\mathcal{A}_j\)}
		\bigl(\PSI(\X,t)\bigr)^\dagger\e_j
		=
		-\e_j\PSI(\X,t),
		\qquad
		(\X,t)\in\BR^n\times[0,T].
	\end{equation}
\end{assumption}

Here \(\e_1,\ldots,\e_n\) are the generators of \(\cl_{0,n}\), and
\(\dagger\) denotes the conjugation in \(\BC\otimes\cl_{0,n}\)
defined in \eqref{conjugation}.

We shall work with the canonical phase $\w$,
which satisfies \(\D\w=\PSI\) and defines the gauge map \(\u=\exp(-i\w)\).
 
  We assume
that this phase commutes with its first-order time and space derivatives. This allows us to differentiate
$\u$ as in the scalar case. The estimates below also
require a uniform condition on the spatial Fourier profile of
\(\PSI\).

\begin{assumption}
	\label{CommutingPhase-Assumption}
	Let $\PSI$ be the solution of \eqref{BO-Dirac} and set
	\[
	\w(\X,t):=\D(-\Delta)^{-1}\PSI(\X,t),
	\qquad (x,t)\in\mathbb R^n\times[0,T].
	\]
	
	We assume that
	\[
	\w\in C^1\bigl(\BR^n\times[0,T];
	\BC\otimes\cl_{0,n}\bigr) \cap 
	L^\infty\bigl([0,T];
	H^\sigma(\BR^n;\BC\otimes \cl_{0,n})\bigr)
	\]
and that
	\begin{equation}\label{CommutingPhase-Condition}
		[\w,\partial_t\w](\X,t)=0,
		\qquad
		[\w,\partial_{x_j}\w](\X,t)=0,
		\quad j=1,\ldots,n,
	\end{equation}
	for every \((\X,t)\in\BR^n\times[0,T]\), where
	\(
	[A,B]:=AB-BA.
	\)
\end{assumption}

\begin{assumption}\label{Dyadic-Assumption}
	Let \(T>0\) and \(\sigma\geq\frac n2-1\). Assume that there are
	constants \(\varepsilon>0\) and \(M>0\) such that
	\begin{equation}\label{low-Frequencies}
		\tag{\(\mathbf{lo}\)}
		\operatorname*{ess\,sup}_{t\in[0,T]}
		\int_{|\XI|\leq1}
		\bigl|\widehat{\PSI}(\XI,t)\bigr|_0^2\,\dxi
		\leq M^2,
	\end{equation}
	and
	\begin{equation}\label{hi-Frequencies}
		\tag{\(\mathbf{hi}\)}
		\operatorname*{ess\,sup}_{t\in[0,T]}
		\int_{2^{j-1}\leq|\XI|\leq2^{j+1}}
		\bigl|\widehat{\PSI}(\XI,t)\bigr|_0^2\,\dxi
		\leq M^2\,2^{-2j(\sigma+1+\varepsilon)},
		\qquad
		j\geq1,
	\end{equation}
	where the constant $M$, appearing in \eqref{low-Frequencies} and \eqref{hi-Frequencies}, is independent of $j$ and $t$.
\end{assumption}

Under Assumption~\ref{AntiCommutation-Assumption}, the canonical phase \(\w\) is
Hermitian; see the proof of Proposition~\ref{Tecnical-Prop2}. Consequently, the
gauge map \(\u=\exp(-i\w)\) constructed from \(\w\) satisfies
\eqref{unit-modulus}. Assumption~\ref{CommutingPhase-Assumption} is imposed directly on
\(\w\). In particular, it provides the regularity and
commutation properties required to differentiate $\u$
without additional commutator terms. Assumption~\ref{Dyadic-Assumption} provides the
uniform bounds on \(\PSI\) required in the estimates for the gauge
nonlinearities in \eqref{Schr-Alfa}.

The following theorem gives the gauge correspondence.

\begin{theorem}\label{Theorem-Gauge}
Let \(\PSI\) solve \eqref{BO-Dirac} and satisfy Assumptions~\ref{AntiCommutation-Assumption}--\ref{Dyadic-Assumption}, and let $\u:\BR^n \times [0,T]\longrightarrow \BC\otimes\cl_{0,n}$ be defined by
	\begin{equation}\label{upm-Gauge}
		\u(\X,t)=\exp\bigl(-i\D(-\Delta)^{-1}\PSI(\X,t)\bigr),\qquad (\X,t)\in \BR^n \times [0,T].
	\end{equation}
	
	Then $\u$ satisfies the unit-modulus property \eqref{unit-modulus}. Moreover, its Plemelj components
	\begin{eqnarray}\label{u-pm}
		\u_\pm(\X,t)=\frac{1}{2}\bigl(\u(\X,t)\pm \mathcal{H}\u(\X,t)\bigr),& (\X,t)\in \BR^n \times [0,T]
	\end{eqnarray}
	solve the Schr\"odinger system \eqref{Schr-Alfa}, with the nonlinearity ${\bm G}(\cdot)$
	defined by~\eqref{Gu}.
\end{theorem}

The next result yields the a priori Sobolev bounds used later.
Under Assumptions \ref{AntiCommutation-Assumption}--\ref{Dyadic-Assumption}, and for solutions of
\eqref{BO-Dirac} satisfying the prescribed time regularity, the homogeneous Sobolev energy is differentiable on
$[0,T]$. The corresponding differential inequality is treated by the nonhomogeneous form of Gronwall's lemma (cf. \cite[Subsection 1.2]{Tao06}).

\begin{theorem}\label{Theorem-Hs}
Let \(\PSI\) solve \eqref{BO-Dirac} and satisfy Assumptions~\ref{AntiCommutation-Assumption}--\ref{Dyadic-Assumption}. If $\displaystyle \sigma \geq \frac{n}{2}-1$ is as in Assumption~\ref{Dyadic-Assumption}, then
	\begin{eqnarray}
		\label{Hs-Condition-Theorem}\PSI(\cdot,t)\in {H}^{\sigma+1}(\mathbb{R}^n;\mathbb{C}\otimes\cl_{0,n})\cap L^\infty(\mathbb{R}^n;\mathbb{C}\otimes\cl_{0,n}),
		\quad \mbox{for all} \quad t\in [0,T],
	\end{eqnarray}
	and
	\begin{eqnarray}
		\label{Mixed-supHs}\|\PSI\|_{L^\infty_T\dot{H}_{\X}^{\sigma+1}}=		\operatorname*{ess\,sup}_{t\in [0,T]}\|\PSI(\cdot,t)\|_{\dot{H}^{\sigma+1}}, &\displaystyle  \|\PSI\|_{L^\infty_TL_{\X}^\infty}=		\operatorname*{ess\,sup}_{t\in [0,T]}\|\PSI(\cdot,t)\|_{L^\infty}.
	\end{eqnarray}
	are finite.
	
	Moreover, if, in addition
	\begin{eqnarray}
	\label{continuityPSI}	\PSI\in
	C\bigl([0,T];H^{\sigma+1}(\BR^n;\BC\otimes\cl_{0,n})\bigr)
	\cap
	C\bigl([0,T];L^\infty(\BR^n;\BC\otimes\cl_{0,n})\bigr),
	\end{eqnarray}
then the homogeneous Sobolev norm satisfies
	\begin{eqnarray}
		\label{Hs-Conservation}
		\|\PSI(\cdot,t)\|_{\dot{H}^\sigma}\lesssim \mathcal{E}_{\sigma,T}(\PSI), &\mbox{for all}
& t\in [0,T],
	\end{eqnarray}
	where
	\begin{eqnarray}\label{Gronwall-Constant}
		\mathcal{E}_{\sigma,T}(\PSI):=		e^{\frac{T}{2}}
		\left(
		\|\PSI(\cdot,0)\|_{\dot{H}^\sigma}
		+
		\sqrt{T}
		\|\PSI\|_{L^\infty_T \dot{H}^{\sigma+1}_{\X}}
		\|\PSI\|_{L^\infty_T L^\infty_{\X}}
		\right).
	\end{eqnarray}
\end{theorem}

\begin{remark}
	Theorem~\ref{Theorem-Hs} gives an a priori estimate for the
	$\dot{H}^\sigma-$norm on $[0,T]$, with at most exponential dependence on
	$T$. It does not, however, yield conservation of the \(\dot H^\sigma\)-seminorm for
	\eqref{BO-Dirac}. This differs from the scalar
	Benjamin--Ono equation ($n=1$ and $\sigma=0$), for which the $L^2-$norm is
	conserved (see, for instance, \cite[Subsection~4.4]{Tao06}).
\end{remark}

These bounds control the nonlinear terms in the Schr\"odinger system \eqref{Schr-Alfa} used below; in particular, they enter  Proposition~\ref{LeibnizSobolev-Theorem}.

We measure the resulting Strichartz estimates by working with the homogeneous seminorm
\(
\|\PSI\|_{L_T^q\dot{W}_x^{\sigma,p}},
\)
which corresponds to the admissible pair $(q,p)$, defined as follows:
\begin{definition}[Schr\"odinger admissible pair]\label{SchroedingerPair}
	For $n\geq 1$, we say that $\left(q,p\right)$ is sharp Schr\"odinger admissible if
	\begin{eqnarray}
		\label{AdmissibleCondition}
		2\leq p,q\leq \infty,~~~ \displaystyle \frac{2}{q}+\frac{n}{p}=\frac{n}{2} & \mbox{and}& (q,p,n)\neq (2,\infty,2).
	\end{eqnarray}
\end{definition}

The next main result will rely on the homogeneous and inhomogeneous Strichartz estimates for the Schr\"odinger equation (including the endpoint estimate of Keel and Tao \cite{KeelTao98}) and the estimates of Ginibre and Velo \cite{GinibreVelo92}. As for the nonlinear terms in \eqref{Schr-Alfa}, their control will follow from the a priori product estimates of Subsection \ref{Estimates-Nonlinear-Subsection}, mainly from Lemma \ref{GrafakosOh-Lemma}.

\begin{theorem}
	\label{StrichartzTheorem-PSI}
Let \(n\geq3\) and let \(\PSI\) solve \eqref{BO-Dirac} and satisfy~
\ref{AntiCommutation-Assumption}--\ref{Dyadic-Assumption}. Suppose that
	\(\displaystyle 
	\sigma\geq \frac{n}{2}-1
	\)
	and that
	\[
	\PSI\in
	C\bigl([0,T];H^{\sigma+1}(\BR^n;\BC\otimes\cl_{0,n})\bigr)
	\cap
	C\bigl([0,T];L^\infty(\BR^n;\BC\otimes\cl_{0,n})\bigr).
	\]

Assume further that
\begin{equation}
	\label{Smallness-Sup-Hcritical}
	\|\PSI\|_{L^\infty_T\dot{H}^{\frac{n}{2}-1}_{\X}}
	\leq
	\frac{1-2^{-(1+\varepsilon)}}{2C_{n,\sigma+1}},
\end{equation}
where \(C_{n,\sigma+1}>0\) denotes the constant in the Sobolev and fractional Leibniz estimates.

	Then, for every sharp Schr\"odinger admissible pair $(q,p)$ in the sense of Definition \ref{SchroedingerPair}, we obtain
	\begin{equation}
		\label{StrichartzEstimate-Hgamma}
		\|\PSI\|_{L_T^q \dot{W}^{\sigma,p}_{\X}}
		\lesssim
2^{\frac{n}{2}+1}\left(1+2\sqrt{T}\right)e^{-\frac{T}{2}}\mathcal{E}_{\sigma,T}(\PSI)
\left(
1+2\mathcal{E}_{\frac{n}{2}-1,T}(\PSI)
\right)^2,
	\end{equation}
	where $\mathcal{E}_{\sigma,T}(\PSI)$ and $\mathcal{E}_{\frac{n}{2}-1,T}(\PSI)$ ($\displaystyle \sigma=\frac{n}{2}-1$) are the constants defined by \eqref{Gronwall-Constant}.
\end{theorem}

\begin{remark}\label{Well-Posedness-Hs}	
	We do not address the well-posedness problem for \eqref{BO-Dirac}. In
	particular, we do not construct solutions satisfying Assumptions \ref{AntiCommutation-Assumption}--\ref{Dyadic-Assumption}
	or prove that these assumptions are preserved by the flow.
\end{remark}

\subsection{Organization of the paper}

The paper is organized as follows. Section~\ref{GeneralSetup} introduces the Clifford
algebra setting, the associated function spaces, and the
Littlewood--Paley decomposition, and recalls the pseudo-differential
description of the Dirac operator and the Riesz--Hilbert transform.
Section~\ref{Technical-Section} develops the estimates used in the gauge argument. In
particular, the dyadic assumptions yield the required regularity of
\(\bm \psi\), while Assumptions~\ref{Dyadic-Assumption} and~\ref{CommutingPhase-Assumption} give the gauge properties and
the comparison and product estimates. Section~\ref{Gauge-Section} proves the gauge
correspondence and the a priori Sobolev estimate. Section~\ref{Strichartz-Section} applies
Strichartz estimates to prove Theorem~\ref{StrichartzTheorem-PSI}.

\subsection{Notations for inequalities}\label{Notations-Conventions}
Throughout the paper, the notation \(A \lesssim B\) (respectively \(A \gtrsim B\)) means that
there exists a positive constant \(C > 0\), independent of $A$ and $B$,
such that \(A \leq C B\) (respectively \(A \geq C B\)). In addition, we write \(A \sim B\) when \(A \lesssim B\) and \(B \lesssim A\).

\section{Preliminaries}\label{GeneralSetup}


\subsection{Clifford algebras}\label{CliffordAlgebras}

Let $\cl_{0,n}$ denote the real Clifford algebra of signature $(0,n)$ generated by the basis $\e_1,\e_2,\ldots,\e_n$ of $\BR^n$. The multiplication in $\cl_{0,n}$ is determined by the anti-commutator relations
\begin{eqnarray}
	\label{CliffordBasis}\e_j\e_k + \e_k\e_j =-2\delta_{jk},  \qquad j,k=1,2,\ldots,n.
\end{eqnarray}

Notice that $\cl_{0,n}$ is an associative algebra of dimension $2^n$, whose identity element is $\e_\emptyset := 1$. Moreover, $\BR$ and $\BR^n$ may both be naturally identified with subspaces of $\cl_{0,n}$. Accordingly, the vectors $\X=(x_1,x_2,\ldots,x_n)$ and $\XI=(\xi_1,\xi_2,\ldots,\xi_n)$ in $\BR^n$ will be represented in $\cl_{0,n}$ as
\begin{eqnarray}
	\label{Clifford-X-XI}\X = \sum_{j=1}^n x_j \e_j \quad \text{and} \quad \XI = \sum_{j=1}^n \xi_j \e_j.
\end{eqnarray}

In view of \eqref{CliffordBasis}, the Euclidean inner product
\[
\langle \X, \XI \rangle := \sum_{j=1}^n x_j\xi_j
\]
can equivalently be written in terms of the anti-commutator of $\X$ and $\XI$:
\[
\langle \X, \XI \rangle = -\frac{1}{2} \Bigl( \X\XI + \XI\X \Bigr).
\]

The Clifford algebra $\cl_{0,n}$ is canonically isomorphic to the exterior algebra
\[
\bigwedge^*(\mathbb{R}^n):= \bigoplus_{r=0}^n \bigwedge^r (\mathbb{R}^n),
\]
via the linear extension of the correspondence
\[
\e_{j_1}\e_{j_2}\cdots \e_{j_r} \mapsto dx_{j_1} \wedge dx_{j_2} \wedge \cdots \wedge dx_{j_r}, \quad 1 \leq j_1 < j_2 < \cdots < j_r \leq n
\]
(see, for instance, \cite[Chapter~4]{VazRocha16}). This identification lets us treat solutions to the Benjamin--Ono-type Dirac equation \eqref{BO-Dirac} in a unified way.

More precisely, let $J = \{j_1,j_2,\ldots,j_r\}$ be a subset of $\{1,2,\ldots,n\}$ of cardinality $|J|= r$. We seek solutions expressed as
\begin{eqnarray}
	\label{CliffordPsi}
	\PSI(\X,t) = \sum_{r=0}^n \sum_{|J|=r} \psi_J(\X,t) \e_J, \qquad (\X,t)\in \BR^n \times [0,T],
\end{eqnarray}
where $\e_J := \e_{j_1}\e_{j_2}\cdots \e_{j_r}$ and each function $\psi_J$ is complex-valued.

There are three canonical automorphisms that keep the Clifford-valued structure of $\BC\otimes\cl_{0,n}$. They are defined as below:
\begin{itemize}
	\item The {\it main involution} $\a \mapsto \a'$ is defined
	recursively via
	\begin{eqnarray}
		\label{involution}
		\begin{array}{lll}
			(\a \b)'=\a'\b', \\[0.2cm]
			(a_J \e_J)'
			=\overline{a_J}~\e_{j_1}'\e_{j_2}'\ldots \e_{j_r}', \\[0.2cm]
			\e_j'=-\e_j.
		\end{array}
	\end{eqnarray}
	\item The {\it reversion} $\a \mapsto \a^*$ is defined recursively
	via
	\begin{eqnarray}
		\label{reversion}
		\begin{array}{lll}
			(\a \b)^*=\b^*\a^*, \\[0.2cm]
			(a_J \e_J)^*
			=a_J~\e_{j_r}^*\ldots \e_{j_2}^*\e_{j_1}^*, \\[0.2cm]
			\e_j^*=\e_j.
		\end{array}
	\end{eqnarray}
	\item The {\it conjugation} $\a \mapsto \a^\dag$ is defined
	recursively via
	\begin{eqnarray}
		\label{conjugation}
		\begin{array}{lll}
			(\a \b)^\dag=\b^\dag\a^\dag, \\[0.2cm]
			(a_J \e_J)^\dag =\overline{a_J}~\e_{j_r}^\dag
			\ldots \e_{j_2}^\dag\e_{j_1}^\dag, \\[0.2cm]
			\e_j^\dag=-\e_j.
		\end{array}
	\end{eqnarray}
\end{itemize}

The conjugation in \eqref{conjugation} is the composition of the main involution \eqref{involution} and the reversion \eqref{reversion}. Namely,
\[
(\a')^*=(\a^*)'=\a^\dagger,
\qquad
\a\in\BC\otimes\cl_{0,n}.
\]
For a Clifford vector \(\X\), as in \eqref{Clifford-X-XI}, this gives
\[
\X^*=\X,
\qquad
\X^\dagger=\X'=-\X.
\]

Also, for \(\a\in\BC\otimes\cl_{0,n}\), the scalar part of
\(\a^\dagger\a\) is non-negative, and it is zero only when
\(\a=0\). In addition,
\[
\a^\dagger\a=\a\a^\dagger.
\]

These observations motivate the definition of the Clifford norm $|\cdot|_0$ given next.

\subsection{Clifford-valued function spaces}\label{CliffordFunctionSpaces}


Given two arbitrary Clifford numbers
$\a,\b\in\BC \otimes\cl_{0,n}$, we define

\begin{eqnarray}
\label{InnerProduct-0}(\a,\b)_{0}
:=2^{n}\,\bigl[~\a^\dagger\,\b~\bigr]_{0},
\qquad
\vert~\a~\vert_{0}
:=\sqrt{(\a,\a)_{0}}
=2^{\frac{n}{2}}\sqrt{\bigl[~\a^\dagger\,\a~\bigr]_{0}},
\end{eqnarray}

where \(\left[~\cdot~\right]_0 : \BC\otimes \cl_{0,n}\;\longrightarrow\;\BC\) denotes the projection operator and $\dagger$ the conjugation operation in $\BC\otimes \cl_{0,n}$.

With the normalization $\left\vert~ \e_{\emptyset}~\right\vert_{0}=2^{\frac{n}{2}}$, the algebra
$\cl_{0,n}$, endowed with inner product
$\left(\cdot,\cdot\right)_{0}$ and norm $\vert\cdot\vert_{0}$, becomes both a real
(or complex) Hilbert space and a Banach algebra.  In particular one
checks the following set of inequalities (see~\cite[Chapter 1]{BDS82}):
\begin{enumerate}
	\item[\bf (i)] 
	(\textbf{Product norm inequality}) $$
	\left\vert~{\bm a} {\bm b}~\right\vert_0 \leq \left\vert~{\bm a}~\right \vert_0 \left\vert~{\bm b}~\right\vert_0,~~{\bm a},{\bm b}\in \BC\otimes\cl_{0,n};
	$$
	\item[\bf (ii)] (\textbf{Triangle Inequality})
	$$\left\vert~{\bm a}  + {\bm b}~\right\vert_0  \leq \left\vert~{\bm a}~\right\vert_0 + \left\vert~{\bm b}~\right\vert_0, ~~{\bm a},{\bm b}\in \BC\otimes\cl_{0,n}.$$ 
\end{enumerate}

For $1\leq p\leq \infty$, we define the \emph{Banach right–module} $L^p(\BR^n;\BC\otimes \cl_{0,n})$
endowed with the norm
\begin{eqnarray}
\label{Lp-Norm} \| \f \|_{L^p}
:= 
\begin{cases}
	\displaystyle
	\left(\int_{\mathbb{R}^n} \left|~\f(\X)~\right|_0^p \dx\right)^{\frac{1}{p}}
	&,~1\le p<\infty,\\
	\displaystyle			\operatorname*{ess\,sup}_{\X\in \BR^n} \left|~\f(\X)~\right|_0&,~p=\infty.
\end{cases}
\end{eqnarray}

Observe that $\f \in L^p(\BR^n;\BC\otimes \cl_{0,n})$ if and only if the real-valued function $\X \mapsto |\f(\X)|_0$ lies in $L^p(\BR^n)$.
In the case $p=2$, $L^p(\BR^n;\BC\otimes \cl_{0,n})$ becomes a unitary right $\BC \otimes\cl_{0,n}-$module with sesquilinear form
\begin{eqnarray}
	\label{InnerP}
	\langle \f, \g\rangle:= 2^n \int_{\mathbb{R}^n} \big(\f(\X)\big)^{\dagger} \g(\X) \dx.
\end{eqnarray} 
and induced norm provided by
\begin{eqnarray}
\label{L2-Norm}\left\langle \f, \f\right\rangle_0 := 2^n \int_{\mathbb{R}^n} \left[~\big(\f(\X)\big)^{\dagger} \f(\X) ~\right]_0 \dx = \int_{\mathbb{R}^n} \vert \f(\X)\vert^2_0 \dx = \|\f \|_{L^2}^2.
\end{eqnarray}

\begin{remark}
The Clifford norm, defined by $\a\mapsto\sqrt{\a^\dagger\a}$, fail to be submultiplicative. In the present setting, however, the submultiplicativity of the Clifford absolute value $|\cdot|_0$ is essential, as it allows one to establish the fundamental properties of the corresponding $L^p$-spaces, including H\"older’s and Young’s inequalities, following~\cite[Remark~4]{BernsteinFaustino25}.
\end{remark}

We write $\mathcal{S}(\BR^n)$ for the Schwartz space and $\mathcal{S}'(\BR^n)$ for the space of tempered distributions. The Clifford-valued counterparts of these function spaces, denoted by
\begin{eqnarray*}
	\mathcal{S}(\BR^n;\BC\otimes \cl_{0,n})&:=&\mathcal{S}(\BR^n)\otimes (\BC\otimes \cl_{0,n}), \\
	\mathcal{S}'(\BR^n;\BC\otimes \cl_{0,n})&:=&\mathcal{S}'(\BR^n)\otimes (\BC\otimes \cl_{0,n}),
\end{eqnarray*} 
respectively,
consist of those $\BC\otimes \cl_{0,n}$‐valued functions $\f$ for which the scalar-valued function $\X\mapsto |\f(\X)|_0$ belongs to $\mathcal{S}(\BR^n)$ resp. $\mathcal{S}'(\BR^n)$.

Now, let
\begin{eqnarray}
	\label{FourierTransform}
	\widehat{\f}(\XI):=(\mathcal{F} \f)(\XI)=\int_{\BR^n} \f(\X)e^{-i \langle \X,\XI \rangle}d\X, & ~\XI\in \BR^n~
\end{eqnarray}
be the Fourier transform over $\BR^n$ and by
\begin{eqnarray}
	\label{FourierInverse}(\mathcal{F}^{-1} \g)(\X)=\frac{1}{(2\pi)^n}\int_{\BR^n} \g(\XI)e^{i \langle \X,\XI \rangle}d\XI, & ~\XI\in \BR^n~
\end{eqnarray}
its inverse.

For \(1\leq p<\infty\), the space \(\mathcal{S}(\BR^n;\BC\otimes\cl_{0,n})\) is dense in \(L^p(\BR^n;\BC\otimes\cl_{0,n})\) (cf. \cite[Proposition 2.4.23]{hytonen2016analysis}). The sesquilinear form \eqref{InnerP} then allows us to extend the topological isomorphism $$\mathcal{F}:\mathcal{S}(\BR^n;\BC\otimes\cl_{0,n})\longrightarrow \mathcal{S}(\BR^n;\BC\otimes \cl_{0,n})$$ to $L^2(\BR^n;\BC\otimes \cl_{0,n})$ by using density arguments, and to $\mathcal{S}'(\BR^n;\BC\otimes\cl_{0,n})$ via the duality relation
\begin{eqnarray*}
	\langle \widehat{\f},\g \rangle=\langle \f,\widehat{\g}\rangle, & \f\in \mathcal{S}(\BR^n;\BC\otimes\cl_{0,n}), & \g\in \mathcal{S}'(\BR^n;\BC\otimes \cl_{0,n}).
\end{eqnarray*}

In particular, Plancherel's identity \cite[Theorem 2.4.9]{hytonen2016analysis}, 
\begin{equation}\label{Plancherel} \|\f\|_{L^2} =\frac{1}{ (2\pi)^{\frac{n}{2}}}\|\widehat{\f}\|_{L^2}, 
\end{equation} 
follows because the normalized Fourier transform \(\displaystyle \f\mapsto \frac{1}{ (2\pi)^{\frac{n}{2}}}\widehat{\f}\) extends from \(\mathcal{S}(\BR^n;\BC\otimes\cl_{0,n})\) to a unitary operator on \(L^2(\BR^n;\BC\otimes\cl_{0,n})\).

Next, for $1\leq p\leq \infty$ and $s\in \BR$, we define by 
\begin{eqnarray}
	\label{Sobolev}	W^{s,p}(\BR^n;\BC\otimes\cl_{0,n}):=\left\{~\f\in \mathcal{S}'(\BR^n;\BC\otimes\cl_{0,n})~:~\left\|~\f~\right\|_{W^{s,p}}<\infty~\right\}
\end{eqnarray}
the Sobolev spaces of $\BC\otimes\cl_{0,n}-$valued functions induced by the norm 
\begin{eqnarray}
	\label{Sobolev-Norm-Wsp} \displaystyle \left\|~\f~\right\|_{{W}^{s,p}} :=\left\|~(I-\Delta)^{\frac{s}{2}}\f~\right\|_{L^p},
\end{eqnarray}
and by
\begin{eqnarray}
	\label{SobolevHomogeneous}	\dot{W}^{s,p}(\BR^n;\BC\otimes\cl_{0,n}):=\left\{~\f\in \mathcal{S}'(\BR^n;\BC\otimes\cl_{0,n})~:~\left\|~\f~\right\|_{\dot{W}^{s,p}}<\infty~\right\}
\end{eqnarray}
the {\it homogeneous Sobolev spaces} of $\BC\otimes\cl_{0,n}-$valued functions induced by the seminorm 
\begin{eqnarray}
	\label{SobolevHomogeneous-Norm-Wsp} \displaystyle \left\|~\f~\right\|_{\dot{W}^{s,p}} :=\left\|~(-\Delta)^{\frac{s}{2}}\f~\right\|_{L^p}.
\end{eqnarray}

Here and in what follows $(I-\Delta)^{\frac{s}{2}}$ and $(-\Delta)^{\frac{s}{2}}$ are the fractional Riesz and Bessel operator of order $s$, respectively defined by the spectral formulas
\begin{eqnarray}
	\label{BesselOperators}
	\mathcal{F}\left((I-\Delta)^{\frac{s}{2}}\f\right)(\XI)=\left(1+|\XI|^2\right)^{\frac{s}{2}} \widehat{\f}(\XI), & \XI\in \BR^n, \\
	\label{RieszOperators}
	\mathcal{F}\left((-\Delta)^{\frac{s}{2}}\f\right)(\XI)=|\XI|^s \widehat{\f}(\XI), & \XI\in \BR^n.
\end{eqnarray}

Thus, these spaces generalize the \emph{Banach right--modules} $L^p(\BR^n;\BC\otimes \cl_{0,n})$ introduced previously, corresponding to the case $s=0$. Moreover, for $1\leq p\leq \infty$ and $\sigma \geq s >0$, one has the inequalities
\begin{eqnarray}
	\label{normIneq-Wsp} \|\f\|_{L^p}\lesssim \|\f\|_{W^{s,p}} ,& \|\f\|_{\dot{W}^{s,p}}\lesssim \|\f\|_{W^{s,p}}, & \|\f\|_{W^{\sigma,p}} \leq \|\f\|_{W^{s,p}},
\end{eqnarray}
and the norm equivalence 
\begin{eqnarray}
	\label{normEquivalence-Wsp} \|\f\|_{W^{s,p}}\sim \|\f\|_{L^p}+\|\f\|_{\dot W^{s,p}}.
\end{eqnarray}

This is a standard consequence of the Fourier-multiplier description of
Bessel potential spaces; see, for instance,
\cite[Subsection 6.2.1]{GrafakosVol250}. More precisely, if
\(0\leq \sigma\leq s\), then the symbols
\((1+|\XI|^2)^{\frac{s}{2}}\), \(|\XI|^s\), and
\((1+|\XI|^2)^{\frac{\sigma}{2}}\), associated with
\((I-\Delta)^{\frac{s}{2}}\), \((-\Delta)^{\frac{s}{2}}\), and
\((I-\Delta)^{\frac{\sigma}{2}}\), respectively, satisfy
\begin{eqnarray*}
	\begin{array}{lllll}
		1 \leq (1+|\XI|^2)^{\frac{s}{2}}, &  |\XI|^s \leq (1+|\XI|^2)^{\frac{s}{2}}, & 	(1+|\XI|^2)^{\frac{\sigma}{2}} \leq (1+|\XI|^2)^{\frac{s}{2}}, &\mbox{for all}&\XI\in \BR^n,
	\end{array}
\end{eqnarray*}
which yields \eqref{normIneq-Wsp}
so that 
\begin{eqnarray*}
	\|\f\|_{L^p}+\|\f\|_{\dot W^{s,p}}\lesssim\|\f\|_{W^{s,p}}.
\end{eqnarray*}

Conversely, the reverse estimate
\begin{eqnarray*}
	\|\f\|_{W^{s,p}}\lesssim	\|\f\|_{L^p}+\|\f\|_{\dot W^{s,p}}
\end{eqnarray*}
follows from a standard Mikhlin multiplier argument.

From now on, we will adopt mostly the notations $\left\|~\cdot~\right\|_{{W}^{s,p}}$ and $\left\|~\cdot~\right\|_{\dot{W}^{s,p}}$ when we refer to the norm of \eqref{Sobolev} and the seminorm of \eqref{SobolevHomogeneous}, respectively.

Also, the $H^s$-norm and $\dot{H}^s-$seminorm notations
\begin{center}
	$\left\|~\cdot~\right\|_{{H}^s}:=\left\|~\cdot~\right\|_{{W}^{s,2}}$ resp. $\left\|~\cdot~\right\|_{\dot{H}^s}:=\left\|~\cdot~\right\|_{\dot{W}^{s,2}}$
\end{center}
will be adopted when we are referring to $\displaystyle {H}^{s}(\BR^n;\BC\otimes\cl_{0,n}):={W}^{s,2}(\BR^n;\BC\otimes \cl_{0,n})$ and $\displaystyle \dot{H}^{s}(\BR^n;\BC\otimes\cl_{0,n}):=\dot{W}^{s,2}(\BR^n;\BC\otimes \cl_{0,n})$, respectively. 

In particular, the set of identities
\begin{eqnarray}
	\label{Sobolev-Hs} \left\|~\f~\right\|_{{H}^{s}}&=&\frac{1}{(2\pi)^{\frac{n}{2}}}\left\|~\left(1+|\cdot|^2\right)^{\frac{s}{2}}\widehat{\f}~\right\|_{L^2}, \\
	\label{Sobolev-Homogeneous-Hs} \left\|~\f~\right\|_{\dot{H}^{s}}&=&\frac{1}{(2\pi)^{\frac{n}{2}}}\left\|~|\cdot|^{s}\widehat{\f}~\right\|_{L^2}
\end{eqnarray}
follow from~\eqref{Plancherel}, \eqref{BesselOperators} and \eqref{RieszOperators}.

\begin{remark}
	Here, \(W^{s,p}\) denotes the Bessel potential space in the sense of
	\cite{Tao06}, with the same convention for
	\(\BC\otimes\cl_{0,n}\)-valued functions. The homogeneous space
	\(\dot{W}^{s,p}\) is defined by the Fourier multiplier
	\(|\XI|^s\), rather than by the Gagliardo seminorm
	\(\left[\,\cdot\,\right]_{W^{s,p}}\); see, for example,
	\cite[Subsection~2.5.d]{hytonen2016analysis}.
\end{remark}

We shall use throughout the paper the following fractional product estimate for
functions taking values in $\BC\otimes\cl_{0,n}$.

\begin{lemma}[Homogeneous fractional Leibniz estimate]
	\label{GrafakosOh-Lemma}
	Let $s\geq0$, and let
	\[
	1<p,q_1,q_2,r_1,r_2\leq\infty
	\]
	satisfy
	\[
	\frac1p
	=
	\frac1{q_1}+\frac1{q_2}
	=
	\frac1{r_1}+\frac1{r_2}.
	\]
	Then, for all
	\[
	\f,\g\in\mathcal{S}(\mathbb{R}^n;\BC\otimes\cl_{0,n}),
	\]
	one has
	\begin{equation}
		\label{LeibnizEstimate2}
		\|\f\g\|_{\dot W^{s,p}}
		\lesssim
		\|\f\|_{\dot W^{s,q_1}}\|\g\|_{L^{q_2}}
		+
		\|\f\|_{L^{r_1}}\|\g\|_{\dot W^{s,r_2}}.
	\end{equation}
\end{lemma}

When $s = 0$,~\eqref{LeibnizEstimate2} reduces to H\"older's inequality.
For $s > 0$, it follows from the fractional product estimates of Kenig--Ponce--Vega~\cite{KenigPonceVega93}, Kato--Ponce inequality of Grafakos-Oh~\cite{GrafakosOh14} and the submultiplicativity of $|\cdot|_0$ obtained above.

Moreover, on the space-time \( \BR^n \times [0,T]\), with $T>0$, we work with the mixed-(semi)norm spaces \[L^q_T B_{\X}:=L^q([0,T],B(\BR^{n};\BC\otimes \cl_{0,n}))\] for \(1 \leq q \leq \infty\), where \(B(\BR^n;\BC\otimes \cl_{0,n})\) denotes one of the Clifford-valued function spaces introduced above. Their (semi)norms are given by
\begin{eqnarray}\label{mixedNorm}
	\| \u\|_{L^q_T B_{\X}} :=	
	\begin{cases}
		\displaystyle
		\left(\int_{0}^T \|\u(\cdot,t)\|_{B}^q \, dt\right)^{\frac{1}{q}},
		& 1 \le q < \infty,\\[1.0ex]
		\displaystyle
		\operatorname*{ess\,sup}_{t\in [0,T]}\|\u(\cdot,t)\|_{B},
		& q=\infty.
	\end{cases}
\end{eqnarray}

\subsection{Littlewood--Paley theory}\label{LittlewoodPaley-Subsection}

We now recall the Littlewood--Paley decomposition used throughout the paper. Let us choose $\eta\in C^\infty_0(\BR^n)$ such that
\begin{eqnarray*}
	\eta(\XI)=\eta(|\XI|), & 0\leq\eta(\XI)\leq1,~\mbox{for all}~ \XI\in\BR^n,
\end{eqnarray*}
and
\begin{eqnarray*}
	\operatorname{supp}(\eta)\subseteq\{\XI\in\BR^n:|\XI|\leq2\}, &
	\eta(\XI)=1~\mbox{for all}~ |\XI|\leq1.
\end{eqnarray*}

For $j\in\BN_0$, set
\begin{eqnarray}
	\label{Dyadic-phi2j}
	\phi_{2^j}(\XI)=\begin{cases}
		\eta(2\XI), &,~j=0,\\
		\eta(2^{-j}\XI)-\eta(2^{-j+1}\XI), &,~j\geq1.
	\end{cases}
\end{eqnarray}

Then
\begin{eqnarray}
	\label{Partition-phi2j}
	\sum_{j=0}^\infty\phi_{2^j}(\XI)=1,
	&\mbox{for all}&\XI\in\BR^n,
\end{eqnarray}
and
\begin{eqnarray}
	\label{supp-lo}
	\operatorname{supp}(\phi_{1})
	\subseteq\{\XI\in\BR^n:|\XI|\leq 1\},
	& \mbox{for}& j=0,
	\\
	\label{supp-hi}
	\operatorname{supp}(\phi_{2^j})
	\subseteq\{\XI\in\BR^n:2^{j-1}\leq|\XI|\leq2^{j+1}\},
	& \mbox{for}& j\geq1.
\end{eqnarray}

Define
\begin{eqnarray}
	\label{LittlewoodPaleyPj}
	P_j\f=\mathcal{F}^{-1}(\phi_{2^j}\widehat{\f}),
	& j\in\BN_0.
\end{eqnarray}

Thus, for $j\geq1$, the Fourier transform of $P_j\f$ is supported in an
annulus of radius $|\XI|\sim 2^j$.

For $1\leq p\leq q\leq\infty$, $s\geq0$, and $j\geq1$, one has the Bernstein estimates
\begin{eqnarray}
	\label{BernsteinEquiv}
	\left\|P_j(-\Delta)^{\pm\frac{s}{2}}\f\right\|_{L^p}
	\sim
	2^{\pm js}\|P_j\f\|_{L^p}.
\end{eqnarray}
and 
\begin{eqnarray}
	\label{BernsteinInequality}
	\|P_j\f\|_{L^q}
	\lesssim
	2^{jn\left(\frac{1}{p}-\frac{1}{q}\right)}
	\|P_j\f\|_{L^p},
	\qquad
	j\in\BN_0.
\end{eqnarray}

The constants in \eqref{BernsteinEquiv} and \eqref{BernsteinInequality} are
independent of $j$ and $\f$.

For $s\geq0$ and $p=2$, Plancherel's identity \eqref{Plancherel} yields the norm equivalence
\begin{eqnarray}
	\label{normEquivalence-LPaley}
	\|\f\|_{H^s}^2
	\sim
	\sum_{j=0}^\infty
	2^{2js}\|P_j\f\|_{L^2}^2.
\end{eqnarray}

Also, taking $p=2$ and $q=\infty$ in
\eqref{BernsteinInequality}, and using \eqref{Partition-phi2j}, gives the norm inequality
\begin{eqnarray}
	\label{normInequality-LPaley}
	\|\f\|_{L^\infty}
	\lesssim
	\sum_{j=0}^\infty
	2^{\frac{jn}{2}}\|P_j\f\|_{L^2}.
\end{eqnarray}

\subsection{Dirac operator vs. Riesz--Hilbert transform}\label{Subsection-FractionalDirac-RieszHilbert}	

We now recall the frameword developed in \cite{BernsteinFaustino25,Faustino24}, following the approach of \cite{Bernstein2016}.

The identity 
\begin{eqnarray}
	\label{Factorization-D}\D^2=-\Delta,
\end{eqnarray} relates the Dirac operator \(\D\), given by \eqref{DiracOp}, to the Laplace operator $\Delta$.
 Both operators act on the
Clifford-valued function and distribution spaces introduced in
Subsection~\ref{CliffordFunctionSpaces}.

 For \(\f\in\mathcal{S}(\BR^n;\BC\otimes\cl_{0,n})\), the Fourier transform of \(\D\f\) is given by 
\begin{eqnarray}
	\label{spectralDLaplace}
	\mathcal{F}(\D \f)(\XI) = i\XI \widehat{\f}(\XI),\quad \XI\in \BR^n.
\end{eqnarray}

Writing $i\XI$ in spherical coordinates as
\begin{eqnarray}
\label{DiracSymbol-Spherical}i\XI=|\XI|\,\frac{i\XI}{|\XI|}, \qquad \XI \in \BR^n\setminus\{\bm 0\},
\end{eqnarray}
we recover the pseudo-differential representation of the Riesz--Hilbert transform $\mathcal{H}$, given by \eqref{RieszHilbertOp}, from its symbol. 

Set
$$
\displaystyle {\bm \omega}:=\frac{i\XI}{\vert \XI\vert}.
$$

Then ${\bm \omega}^2=1$, and Plancherel's identity \eqref{Plancherel} shows that $\mathcal{H}$ is an $L^2-$isometry. Since
$
\mathcal{H}^2=I,
$
it is surjective and hence unitary on
\(L^2(\mathbb R^n;\mathbb C\otimes C\ell_{0,n})\).

By the Fourier convolution theorem and the ansatz
\begin{eqnarray*}
	\label{RieszHilbertLinearCombination} \mathcal{H}=\sum_{j=1}^{n}\e_j R_j, \quad R_j:=\mathcal{F}^{-1}\frac{i\xi_j}{\vert \XI\vert} \mathcal{F}, 
\end{eqnarray*}
the Riesz transforms $R_j$ admit the kernel representation (cf.\ \cite[p.~224]{SteinWeiss72})
\begin{equation}\label{RieszKernel}
	E_{j}(\X)=\frac{\Gamma\bigl(\tfrac{n+1}{2}\bigr)}{\pi^{\tfrac{n+1}{2}}}\;\frac{x_{j}}{|\X|^{\,n+1}},\quad \X\in\mathbb{R}^{n}\setminus\{0\},
\end{equation}
and we obtain the singular integral formula
\begin{eqnarray} \label{RieszHilbertSingularInt} 
	\mathcal{H} \f(\X)=\frac{\Gamma\left(\frac{n+1}{2}\right)}{\pi^{\frac{n+1}{2}}} \, P.V. \int_{\BR^n}\frac{\X-\Y}{\vert \X-\Y\vert^{n+1}}\f(\Y)\dy, \quad \X \in \BR^n. 
\end{eqnarray}

Since $(i\XI)^2=|\XI|^2$, duality arguments yield, in $\mathcal{S}'(\mathbb{R}^n;\BC\otimes \cl_{0,n})$, the spectral formula
\begin{align}
	\label{SpaceFractionalDirac}
	\D=(-\Delta)^{\frac{1}{2}}\mathcal{H}.
\end{align}

\begin{remark}
	The operators $\D$, $(-\Delta)^{\frac{s}{2}}$, and $\mathcal{H}$ are defined on $\mathcal{S}(\mathbb{R}^n;\mathbb{C}\otimes\cl_{0,n})$ through their symbols and extend continuously to $\mathcal{S}'(\mathbb{R}^n;\mathbb{C}\otimes\cl_{0,n})$, as recalled in Subsection~\ref{CliffordFunctionSpaces}. We work first on $\mathcal{S}(\mathbb{R}^n;\mathbb{C}\otimes\cl_{0,n})$ and then extend the arguments by duality to $\mathcal{S}'(\mathbb{R}^n;\mathbb{C}\otimes\cl_{0,n})$. Throughout, the same convention is used for the $L^p$-spaces and the Sobolev spaces $W^{s,p}(\mathbb{R}^n;\mathbb{C}\otimes\cl_{0,n})$.
\end{remark}

Now, let us look to the so-called Plemelj projection operators given by
\begin{eqnarray}
	\label{ProjOp-RiezHilbert} \P_\pm =\frac{1}{2}\left(I\pm\mathcal{H}\right).
\end{eqnarray}

In terms of Fourier multipliers, one has
\begin{eqnarray}
	\label{ProjectionOp-pm}
		\P_+ \f = \mathcal{F}^{-1} \left(\chi_+(\cdot)\widehat{\f}\right)&\mbox{and}&
		\P_- \f = \mathcal{F}^{-1} \left(\chi_-(\cdot)\widehat{\f}\right),
	\end{eqnarray}
	where
	\begin{eqnarray}
		\label{ProjectionOp}
		\chi_\pm(\XI)=\frac{1}{2}\left(1\pm  \frac{i\XI}{\vert \XI\vert}\right), \quad \XI\in \BR^n\setminus \{\bm 0\}.
	\end{eqnarray}
	
	The identities
	\begin{eqnarray}
		\label{ProjectionOpProperties}
		\begin{array}{lll}
			\chi_+(\XI)+\chi_-(\XI)=1, &  \bigl(\chi_+(\XI)\bigr)^2=\chi_+(\XI), & \bigl(\chi_-(\XI)\bigr)^2=\chi_-(\XI),   \\[0.2cm]
			\chi_+(\XI)\chi_-(\XI)=0, & \chi_-(\XI)\chi_+(\XI)=0,
		\end{array}
	\end{eqnarray}
	(cf.~\cite[p. 671]{li1994clifford} and \cite[Subsection 5.2.1]{mcintosh1996clifford}) show that $\P_\pm$, defined by \eqref{ProjOp-RiezHilbert}, are complementary projection operators. In particular, they satisfy
	\begin{eqnarray}
		\label{ProjectionOp-Lp}	\P_\pm:L^p(\BR^n;\BC\otimes\cl_{0,n})\longrightarrow L^p(\BR^n;\BC\otimes\cl_{0,n}), \qquad 1<p<\infty,
	\end{eqnarray}
	and provide the basic decomposition associated with the Riesz–Hilbert transform that underlies the gauge correspondence between the solutions of \eqref{BO-Dirac} and \eqref{Schr-Alfa}.
	
	Since \(\P_\pm\) commute with scalar Fourier multipliers, the boundedness
	in \eqref{ProjectionOp-Lp} extends to \(\dot{W}^{s,p}\) and \(W^{s,p}\) for
	every \(s\in\BR\) and \(1<p<\infty\).
	
\section{Technical Results}\label{Technical-Section}


\subsection{Consequences of Assumption \ref{Dyadic-Assumption}}

The low and high frequency conditions \eqref{low-Frequencies} and
\eqref{hi-Frequencies} provide uniform dyadic control of the corresponding
Littlewood--Paley pieces. Summing over the dyadic scales yields the following
embedding-type result, which is stated in the next proposition.

\begin{proposition}\label{Tecnical-Prop1}
	Let $\PSI$ solve \eqref{BO-Dirac} and satisfy Assumption~\ref{Dyadic-Assumption}. Then, 
	\begin{eqnarray}
		\label{Sobolev-Condition}\PSI(\cdot,t)\in {H}^{\sigma+1}(\mathbb{R}^n;\mathbb{C}\otimes\cl_{0,n})\cap L^\infty(\mathbb{R}^n;\mathbb{C}\otimes\cl_{0,n}),& \mbox{for all}& t\in[0,T].
	\end{eqnarray}
	
	Also, for every $t\in [0,T]$ 
	\begin{eqnarray}
		\label{SquezeThm-Condition1} \PSI(\cdot,t)~~\text{is continuous on}~\BR^n& \mbox {such that}&
		\lim_{|\X|\rightarrow \infty} \PSI(\X,t)=0,
	\end{eqnarray}
	where $\displaystyle \sigma\geq \frac{n}{2}-1$ is the constant provided by Assumption~\ref{Dyadic-Assumption}.
	
\end{proposition}

\begin{proof}
	
	\medskip 
	
	\underline{Proof of \eqref{Sobolev-Condition}.}
	
	\medskip 
	
	By \eqref{normInequality-LPaley} and Bernstein's inequality, we have
	\begin{eqnarray}
		\label{Hs-Condition} \|\PSI(\cdot,t)\|_{{H}^{\sigma+1}}
		\lesssim
		\sum_{j=0}^\infty 2^{j(\sigma+1)}\|P_j\PSI(\cdot,t)\|_{L^2},
	\end{eqnarray}
	and
	\begin{eqnarray}
		\label{Linfty-Condition} \|\PSI(\cdot,t)\|_{L^\infty}
		\lesssim
		\sum_{j=0}^\infty 2^{\frac{jn}{2}}\|P_j\PSI(\cdot,t)\|_{L^{2}}.
	\end{eqnarray}
	
	By Assumption~\ref{Dyadic-Assumption}, Plancherel's identity \eqref{Plancherel}, and \eqref{LittlewoodPaleyPj}, we obtain
	\begin{eqnarray}
		\label{L2-phi2j-condition}
		(2\pi)^{n}\operatorname*{ess\,sup}_{t\in [0,T]}\|P_j\PSI(\cdot,t)\|_{L^2}^2
		&=& 
		\operatorname*{ess\,sup}_{t\in [0,T]}\int_{\mathbb{R}^n}\phi_{2^j}(\XI)\,|\widehat{\PSI}(\XI,t)|_0^2\,\dxi \nonumber \\
		&\lesssim&
		M^2\,2^{-2j(\sigma+1+\varepsilon)},\quad j\in \BN_0,
	\end{eqnarray}
	for some \(\varepsilon>0\) and \(\displaystyle M> 0\), where \((\phi_{2^j})_{j\in\mathbb{N}_0}\) is the dyadic family in \eqref{Dyadic-phi2j}. Hence
	\[
	\operatorname*{ess\,sup}_{t\in [0,T]}\|P_j\PSI(\cdot,t)\|_{L^2}
	\lesssim
	M\,2^{-j(\sigma+1+\varepsilon)},\qquad j\in \BN_0.
	\]
	
	Inserting this estimate into \eqref{Hs-Condition}, we obtain
	\begin{eqnarray*}
		\operatorname*{ess\,sup}_{t\in [0,T]}\|\PSI(\cdot,t)\|_{H^{\sigma+1}}
		&\lesssim&
		M\sum_{j=0}^{\infty}2^{j(\sigma+1)}2^{-j(\sigma+1+\varepsilon)} \\
		&=&
		M\sum_{j=0}^{\infty}2^{-j\varepsilon}<\infty.
	\end{eqnarray*}
	
	Since $\displaystyle \sigma+1\geq \frac{n}{2}$, \eqref{Linfty-Condition} gives
	\begin{eqnarray*}
		\operatorname*{ess\,sup}_{t\in [0,T]}\|\PSI(\cdot,t)\|_{L^\infty}
		&\lesssim&
		M\sum_{j=0}^{\infty}2^{\frac{jn}{2}}2^{-j(\sigma+1+\varepsilon)} \\
		&\lesssim&
		M\sum_{j=0}^{\infty}2^{-j\varepsilon}<\infty.
	\end{eqnarray*}
	
	Hence \eqref{Sobolev-Condition} follows.
	
	\medskip 
	
	\underline{Proof of \eqref{SquezeThm-Condition1}.}
	
	\medskip 
	
	By the support conditions \eqref{supp-lo} and \eqref{supp-hi}, the volume of $\operatorname{supp}(\phi_{2^j})$ is bounded above by a multiple of $2^{jn}$. Thus
	\[
	\|\phi_{2^j}(\cdot)\|_{L^2}^2
	=
	\int_{\BR^n} \big(\phi_{2^j}(\XI)\big)^2\dxi
	\lesssim
	2^{jn}.
	\]
	
	The Cauchy--Schwarz inequality gives
	\begin{eqnarray*}
		\int_{\mathbb{R}^n} \phi_{2^j}(\XI)|\widehat{\PSI}(\XI,t)|_0\,\dxi
		&\leq&
		\|\phi_{2^j}(\cdot)\|_{L^2}
		\|\phi_{2^j}(\cdot)\widehat{\PSI}(\cdot,t)\|_{L^2} \nonumber \\
		&\lesssim&
		2^{\frac{jn}{2}}\|\phi_{2^j}(\cdot)\widehat{\PSI}(\cdot,t)\|_{L^2}.
	\end{eqnarray*}
	
	By \eqref{L2-phi2j-condition}, we therefore obtain
	\begin{eqnarray}
		\label{L1-condition}
		\int_{\mathbb{R}^n} \phi_{2^j}(\XI)|\widehat{\PSI}(\XI,t)|_0\,\dxi
		&\lesssim&
		M2^{-j(\sigma+1+\varepsilon-\frac{n}{2})}.
	\end{eqnarray}
	
	Since $\displaystyle \sigma\geq \frac{n}{2}-1$ and $\varepsilon>0$, the series
	\[
	\sum_{j=0}^\infty 2^{-j(\sigma+1+\varepsilon-\frac{n}{2})}
	\]
	is convergent. Hence, by the partition of unity and the triangle inequality,
	\[
	\int_{\mathbb{R}^n}|\widehat{\PSI}(\XI,t)|_0\,\dxi
	\leq
	\sum_{j=0}^\infty
	\int_{\mathbb{R}^n}\phi_{2^j}(\XI)|\widehat{\PSI}(\XI,t)|_0\,\dxi
	<\infty.
	\]
	Thus,
	\[
	\widehat{\PSI}(\cdot,t)\in L^{1}(\BR^n;\BC\otimes\cl_{0,n}).
	\]
	
	The Fourier inversion formula and the Riemann--Lebesgue lemma yield \eqref{SquezeThm-Condition1} (see, for instance, \cite[Lemma 2.4.3]{hytonen2016analysis}).
	
\end{proof}

\begin{remark}\label{Besov-Remark}
Proposition~\ref{Tecnical-Prop1} shows that in the absence of Assumption~\ref{Dyadic-Assumption} a Besov-type condition is needed instead. Using the Littlewood--Paley projectors $P_j$ from Subsection~\ref{LittlewoodPaley-Subsection}, this means that
	\[
	\operatorname*{ess\,sup}_{t\in [0,T]}\sum_{j=0}^\infty 2^{(\sigma+1+\varepsilon)j}\,\|P_j\PSI(\cdot,t)\|_{L^2}<\infty.
	\]
	
That is, $\PSI(\cdot,t) \in B^{\sigma+1+\varepsilon}_{2,1} (\mathbb{R}^n;\BC \otimes \cl_{0,n})$ uniformly in $t\in[0,T]$.

The Besov regularity $\PSI(\cdot,t)\in B^{\sigma+1+\varepsilon}_{2,1}(\mathbb{R}^n;\BC \otimes \cl_{0,n})$ for $t \in [0,T]$ is stronger than the corresponding Sobolev regularity, and yields, for $\displaystyle \sigma\geq \frac{n}{2}-1$ and $\varepsilon>0$, the embedding
	\begin{eqnarray*}
		B^{\sigma+1+\varepsilon}_{2,1}(\mathbb{R}^n;\BC\otimes \cl_{0,n})&\hookrightarrow& H^{\sigma+1}(\mathbb{R}^n;\BC\otimes \cl_{0,n}) \cap L^\infty(\mathbb{R}^n;\BC\otimes \cl_{0,n}).
	\end{eqnarray*}
	
\end{remark}

The remainder of the paper does not make use of Besov spaces. We avoided them to simplify the presentation. For further details, see e.g. Grafakos~\cite[Subsection~6.5]{GrafakosVol250}.
	
\subsection{The gauge map}

Now let
\begin{equation}\label{Gauge-u}
	\u(\X,t)=\exp(-i\w(\X,t)),
	\qquad
	(\X,t)\in\BR^n\times[0,T].
\end{equation}

The commutation condition in
Assumption~\ref{CommutingPhase-Assumption} allows us to
differentiate this expression in the usual way. We record the
resulting identities for later use.

\begin{lemma}
	\label{Technical-Lemma0}
	Suppose that \(\w\) satisfies
	Assumption~\ref{CommutingPhase-Assumption}. Then, $\u$ given by \eqref{Gauge-u} satisfies
	\begin{equation}\label{dt-u}
		\partial_t\u(\X,t)
		=
		-i(\partial_t\w(\X,t))\u(\X,t),
		\qquad
		(\X,t)\in\BR^n\times[0,T],
	\end{equation}
	and
	\begin{equation}\label{Dirac-u}
		\D\u(\X,t)
		=
		-i(\D\w(\X,t))\u(\X,t),
		\qquad
		(\X,t)\in\BR^n\times[0,T].
	\end{equation}
\end{lemma}

\begin{proof}
	Let \(\partial_\mu\) denote either \(\partial_t\) or
	\(\partial_{x_j}\), \(j=1,\ldots,n\). By the power series expansion
	of \eqref{Gauge-u}, we have
	\[
	\u(\X,t)
	=
	\sum_{k=0}^{\infty}
	\frac{(-i)^k}{k!}\bigl(\w(\X,t)\bigr)^k.
	\]
	
Differentiation term by term gives
	\[
	\partial_\mu\u(\X,t)
	=
	\sum_{k=1}^{\infty}
	\frac{(-i)^k}{k!}\partial_\mu\bigl(\w(\X,t)\bigr)^k.
	\]
	
	For $k\in \BN$, the product rule yields
	\[
	\partial_\mu\bigl(\w(\X,t)\bigr)^k
	=
	\sum_{\ell=0}^{k-1}
	\bigl(\w(\X,t)\bigr)^\ell
	(\partial_\mu\w(\X,t))
	\bigl(\w(\X,t)\bigr)^{k-1-\ell},
	\]
	and, according to Assumption~\ref{CommutingPhase-Assumption}, the right-hand side of the previous formula simplifies to
	\[
	\partial_\mu\bigl(\w(\X,t)\bigr)^k
	=
	k(\partial_\mu\w(\X,t))\bigl(\w(\X,t)\bigr)^{k-1}.
	\]
	
	Consequently,
	\begin{align*}
		\partial_\mu\u(\X,t)
		&=
		-i(\partial_\mu\w(\X,t))
		\sum_{k=1}^{\infty}
		\frac{(-i)^{k-1}}{(k-1)!}
		\bigl(\w(\X,t)\bigr)^{k-1} \\
		&=
		-i(\partial_\mu\w(\X,t))\u(\X,t).
	\end{align*}
	
Taking \(\partial_\mu=\partial_t\), we obtain \eqref{dt-u}. In the same order of ideas, it readily follows
	\[
	\partial_{x_j}\u(\X,t)
	=
	-i(\partial_{x_j}\w(\X,t))\u(\X,t),
	\qquad j=1,\ldots,n.
	\]
	
Moreover, using \eqref{DiracOp} we obtain
\begin{align*}
	\D\u(\X,t)
	&=-i\sum_{j=1}^n
	\e_j(\partial_{x_j}\w(\X,t))\u(\X,t) \\
	&=-i(\D\w(\X,t))\u(\X,t).
\end{align*}

This proves \eqref{Dirac-u}.
\end{proof}

The next lemma gives the Hermitian identity for \(\D\PSI\) implied by
Assumption~\ref{AntiCommutation-Assumption}. Proposition~\ref{Tecnical-Prop2} then transfers this identity to the canonical phase \(\w\) of \(\u=\exp(-i\w)\).

\begin{lemma}\label{Tecnical-Lemma1}
	Let $\PSI$ solve \eqref{BO-Dirac} and satisfy Assumption~\ref{AntiCommutation-Assumption}.
	
	Then
	\begin{equation}
		\label{Hermitian-Condition}
		\big(\D \PSI(\X,t)\big)^\dagger=\D \PSI(\X,t), \qquad (\X,t)\in \BR^n \times [0,T],
	\end{equation}
\end{lemma}

\begin{proof}
	Since differentiation commutes with $\dagger$, we have
	\[
	\big(\partial_{x_j}\PSI(\X,t)\big)^\dagger
	=
	\partial_{x_j}\big(\PSI(\X,t)\big)^\dagger,
	\qquad j=1,\dots,n.
	\]
	
	Hence, using Assumption~\ref{AntiCommutation-Assumption} we obtain
	\[
	\big(\D\PSI(\X,t)\big)^\dagger
	=
	\left(\sum_{j=1}^n \e_j\,\partial_{x_j}\PSI(\X,t)\right)^\dagger
	=
	-	\sum_{j=1}^n \partial_{x_j}\big(\PSI(\X,t)\big)^\dagger\,\e_j.
	\]
	
	Applying again Assumption~\ref{AntiCommutation-Assumption} at the level of derivatives (which preserves the anti-commutation), we may rewrite
	\[
	\partial_{x_j}\big(\PSI(\X,t)\big)^\dagger\,\e_j
	=
	-\e_j\,\partial_{x_j}\PSI(\X,t),
	\qquad j=1,\dots,n,
	\]
	and therefore
	\[
	\big(\D\PSI(\X,t)\big)^\dagger
	=
	\sum_{j=1}^n \e_j\,\partial_{x_j}\PSI(\X,t)
	=
	\D\PSI(\X,t),
	\]
	which is exactly \eqref{Hermitian-Condition}.
\end{proof}

\begin{proposition}\label{Tecnical-Prop2}
Let $\PSI$ solve \eqref{BO-Dirac} and satisfy Assumptions~\ref{AntiCommutation-Assumption}--\ref{Dyadic-Assumption} and set 	
	\begin{eqnarray}
		\label{w} \w(\X,t):= \D(-\Delta)^{-1}\PSI(\X,t)&,&(\X,t)\in \BR^n \times [0,T].
	\end{eqnarray} 
	
	Then 
\begin{equation}\label{Dw}
	\D \w(\X,t)=\PSI(\X,t),\qquad (\X,t)\in \BR^n \times [0,T],
\end{equation}
	
	Moreover, $\u$, given by \eqref{Gauge-u},
	 satisfies \eqref{unit-modulus} and
	\begin{equation} 
		\label{DuPsi} \D\u(\X,t) = -i\PSI(\X,t)\u(\X,t), \qquad (\X,t)\in\BR^n\times[0,T].
	\end{equation} 
\end{proposition}

\begin{proof}
For $\XI\neq {\bm 0}$, the Fourier symbol of \(\D(-\Delta)^{-1}\) is $
i\XI|\XI|^{-2}$. 

Hence,
\[
(i\XI)\bigl(i\XI|\XI|^{-2}\bigr)=1,
\quad \mbox{for all} \quad \XI\in \BR^n \setminus \{{\bm 0}\},
\]
and therefore
\(\displaystyle 
\D\w(\cdot,t)=\PSI(\cdot,t)
\)
holds in \(\mathcal{S}'(\BR^n;\BC\otimes \cl_{0,n})\).

It was noted in \eqref{Factorization-D} that the operators \(\D\) and \((-\Delta)^{-1}\) commute on \(\mathcal S'(\mathbb R^n;\mathbb C\otimes C\ell_{0,n})\) so that 
\begin{eqnarray}
\label{w-commute}\w(\X,t)=(-\Delta)^{-1}\D\PSI(\X,t) , \qquad (\X,t)\in\BR^n\times[0,T].
\end{eqnarray}

Moreover,
\((-\Delta)^{-1}\) is a real-valued operator and therefore commutes with
the Clifford conjugation $\dagger$, given by \eqref{conjugation}.
 Hence, by Lemma \ref{Tecnical-Lemma1}
\[
\bigl((-\Delta)^{-1}\D\PSI(\X,t)\bigr)^\dagger
=
(-\Delta)^{-1}\bigl(\D\PSI(\X,t)\bigr)^\dagger
=
(-\Delta)^{-1}\D\PSI(\X,t).
\]

Thus, by \eqref{w-commute} we conclude that 
\begin{eqnarray*}
\bigl(\w(\X,t)\bigr)^\dagger=\w(\X,t), &\mbox{for all} &(\X,t)\in \BR^n\times [0,T].
\end{eqnarray*}

Consequently, by the properties of the exponential function
\[
\u(\X,t) \bigl(\u(\X,t)\bigr)^\dagger
=
\exp\left(-i\w(\X,t)+i\bigl(\w(\X,t)\bigr)^\dagger\right)
=
1,
\]
which proves \eqref{unit-modulus}.

Finally, Assumption \eqref{CommutingPhase-Assumption} and Lemma \ref{Technical-Lemma0} give \eqref{DuPsi}.
This proves the proposition.
\end{proof}

We next record two estimates connecting $\u$ and $\PSI$.
Since $\u$ satisfies the unit-modulus property \eqref{unit-modulus} and the identity \eqref{DuPsi}, the fractional Leibniz rule in
Lemma~\ref{GrafakosOh-Lemma} gives the required estimate.  One term still
contains the top-order seminorm of $\u$, but a smallness condition
on $M$ in Assumption~\ref{Dyadic-Assumption} allows it to be moved to
the left-hand side.

For the converse estimate, we rewrite $\PSI$ through
$\D\u$ using \eqref{DuPsi} and \eqref{unit-modulus}.  The fractional
Leibniz rule from Lemma~\ref{GrafakosOh-Lemma}, together with a 
Sobolev embedding argument, then shows that
the $\dot{W}^{s,p}$ seminorm of $\PSI(\cdot,t)$ is controlled by the
$\dot{W}^{s+1,p}$ seminorm of $\u(\cdot,t)$ and the critical quantity
$\|\PSI(\cdot,t)\|_{\dot{H}^{\frac{n}{2}-1}}$.

\begin{proposition}\label{Tecnical-Prop3}
	For every \(t\in[0,T]\), one has
	\begin{equation}
		\label{Linfty-u}
		\|\u(\cdot,t)\|_{L^\infty}=2^{\frac n2}.
	\end{equation}
Moreover, let \(n\geq 3\) and \(s\geq 0\). Suppose that the constants
\(\varepsilon\) and \(M\) in Assumption~\ref{Dyadic-Assumption} satisfy
\begin{equation}
	\label{M-Condition}
	M\leq\frac{1-2^{-(1+\varepsilon)}}{2C_{n,s}},
\end{equation}
where \(C_{n,s}>0\) denotes the constant in the Sobolev and fractional
Leibniz estimates. Then, for every \(t\in[0,T]\),
\begin{equation}
	\label{Hs-u}
	\|\u(\cdot,t)\|_{\dot{H}^{s+1}}
	\lesssim
	2^{\frac{n}{2}+1}\|\PSI(\cdot,t)\|_{\dot{H}^s}.
\end{equation}
\end{proposition}

\begin{proof}
	Since $\displaystyle |1|_0=|\e_\emptyset|_0=2^{\frac{n}{2}}$, the unit-modulus property
	\eqref{unit-modulus}, established in
	Proposition~\ref{Tecnical-Prop2}, gives
	\begin{equation}
		\label{unit-modulus-u}
		|\u(\X,t)|_0=2^{\frac{n}{2}},
		\qquad
		(\X,t)\in\BR^n\times[0,T].
	\end{equation}
	Thus, \eqref{Linfty-u} follows from \eqref{Lp-Norm}.
	
	For the proof of \eqref{Hs-u}, we recall the formula \eqref{DuPsi} that results from Proposition~\ref{Tecnical-Prop2}.	
	Using \eqref{Sobolev-Homogeneous-Hs} and \eqref{DiracSymbol-Spherical}, we obtain
	\[
	\|\u(\cdot,t)\|_{\dot H^{s+1}}
	=
	\|\D\u(\cdot,t)\|_{\dot H^s}
	=
	\|\PSI(\cdot,t)\u(\cdot,t)\|_{\dot H^s}.
	\]
	
	Applying Lemma~\ref{GrafakosOh-Lemma} with
	\[
	p=q_1=2,
	\qquad
	q_2=\infty,
	\qquad
	r_1=n,
	\qquad
	r_2=\frac{2n}{n-2},
	\]
	we obtain
	\begin{align}
		\|\u(\cdot,t)\|_{\dot H^{s+1}}
		\lesssim
		\|\PSI(\cdot,t)\|_{\dot H^s}
		\|\u(\cdot,t)\|_{L^\infty}
		+
		\|\PSI(\cdot,t)\|_{L^n}
		\|\u(\cdot,t)\|_{\dot W^{s,\frac{2n}{n-2}}}.
		\label{LeibnizEstimate-uHs}
	\end{align}
	
	For \(n\geq3\), the Hardy--Littlewood--Sobolev inequality
	(see~\cite[Theorem 6.1.3]{GrafakosVol250}) gives rise to
	\begin{eqnarray*}
		\|\f\|_{L^{n}}
		\lesssim
		\|\f\|_{\dot{H}^{\frac{n}{2}-1}}
		&\mbox{and}&
		\|\g\|_{L^{\frac{2n}{n-2}}}
		\lesssim
		\|\g\|_{\dot{H}^1}.
	\end{eqnarray*}
	
	Applying these estimates to \(\f=\PSI(\cdot,t)\) and
	\(\g=(-\Delta)^{\frac{s}{2}}\u(\cdot,t)\), we obtain from
	\eqref{SobolevHomogeneous-Norm-Wsp}
	\begin{eqnarray*}
		\| \PSI(\cdot,t)\|_{L^n}
		\lesssim
		\| \PSI(\cdot,t)\|_{\dot{H}^{\frac{n}{2}-1}}
		&\mbox{and}&
		\|\u(\cdot,t)\|_{\dot{W}^{s,\frac{2n}{n-2}}}
		\lesssim
		\|\u(\cdot,t)\|_{\dot{H}^{s+1}}.
	\end{eqnarray*}
	
	Therefore, by \eqref{Linfty-u},
	\begin{equation}
		\label{LeibnizEstimate-uHs-2}
		\|\u(\cdot,t)\|_{\dot H^{s+1}}
		\lesssim
		2^{\frac{n}{2}}\|\PSI(\cdot,t)\|_{\dot H^s}
		+
		\| \PSI(\cdot,t)\|_{\dot{H}^{\frac{n}{2}-1}}
		\|\u(\cdot,t)\|_{\dot H^{s+1}}.
	\end{equation}
	
	The Littlewood--Paley decomposition and Bernstein's inequality
	\eqref{BernsteinInequality} give
	\[
	\| \PSI(\cdot,t)\|_{\dot{H}^{\frac{n}{2}-1}}
	\lesssim
	\sum_{j=0}^\infty
	2^{j\left(\frac{n}{2}-1\right)}
	\|P_j\PSI(\cdot,t)\|_{L^2},
	\]
	where \(P_j\) denotes the Littlewood--Paley projection operator
	\eqref{LittlewoodPaleyPj}.
	
	By Assumption~\ref{Dyadic-Assumption}, we obtain
	\[
	\|P_j\PSI(\cdot,t)\|_{L^2}
	\lesssim
	M\,2^{-j\left(\sigma+1+\varepsilon\right)}.
	\]
	
	Hence
	\[
	\| \PSI(\cdot,t)\|_{\dot{H}^{\frac{n}{2}-1}}
	\lesssim
	M\sum_{j=0}^\infty
	2^{-j\left(\sigma+2+\varepsilon-\frac{n}{2}\right)}.
	\]
	
	Since \(\displaystyle \sigma\geq \frac{n}{2}-1\), it follows that
	\[
	\| \PSI(\cdot,t)\|_{\dot{H}^{\frac{n}{2}-1}}
	\lesssim
	M\sum_{j=0}^\infty 2^{-j(1+\varepsilon)}
	=
	\frac{M}{1-2^{-(1+\varepsilon)}},
	\]
	uniformly for \(t\in[0,T]\).

	By the preceding bound and \eqref{M-Condition}, we have
	\[
	C_{n,s}\|\PSI(\cdot,t)\|_{\dot{H}^{\frac{n}{2}-1}}
	\leq
	\frac{C_{n,s}M}{1-2^{-(1+\varepsilon)}}
	\leq \frac{1}{2}.
	\]
	
	Hence, the last term on the right-hand side of
	\eqref{LeibnizEstimate-uHs-2} may be absorbed into the left-hand side. Therefore,
	\[
	\|\u(\cdot,t)\|_{\dot{H}^{s+1}}
	\lesssim
	\frac{2^{\frac{n}{2}}}{1-C_{n,s}\| \PSI(\cdot,t)\|_{\dot{H}^{\frac{n}{2}-1}}}
	\|\PSI(\cdot,t)\|_{\dot{H}^s}
	\leq
	2^{\frac{n}{2}+1}\|\PSI(\cdot,t)\|_{\dot{H}^s},
	\]
	which proves \eqref{Hs-u}.
\end{proof}

\begin{proposition}
	\label{Tecnical-Lemma2}
Let \(n\geq3\), \(1<p<n\), and \(s\geq0\). Let $\PSI$ solve \eqref{BO-Dirac} and satisfy
	Assumptions~\ref{AntiCommutation-Assumption}--\ref{Dyadic-Assumption}, and let \(\u\) be defined in
	Proposition~\ref{Tecnical-Prop2}. Then, for every \(t\in[0,T]\),
	\begin{equation}
		\label{Psi-from-u-Sobolev}
		\|\PSI(\cdot,t)\|_{\dot{W}^{s,p}}
		\lesssim
		2^{\frac{n}{2}}\|\u(\cdot,t)\|_{\dot{W}^{s+1,p}}
		\left(
		1
		+
		\|\PSI(\cdot,t)\|_{\dot{H}^{\frac{n}{2}-1}}
		\right).
	\end{equation}
\end{proposition}

\begin{proof}
	Fix \(t\in[0,T]\). By \eqref{DuPsi} and the unit-modulus property \eqref{unit-modulus}, we have
	\begin{eqnarray*}
		-i\PSI(\X,t)=\D\u(\X,t)\bigl(\u(\X,t)\bigr)^\dagger,
		& \mbox{for all} &
		(\X,t)\in \BR^n\times [0,T].
	\end{eqnarray*}
	
We observe that the Clifford conjugation, defined by \eqref{conjugation}, preserves the Clifford norm $|\cdot|_0$ (see \eqref{InnerProduct-0}) and commutes
	with scalar-valued Fourier multipliers, so that it preserves the homogeneous
	Sobolev seminorms \eqref{SobolevHomogeneous-Norm-Wsp}.
	
Then, by applying Lemma~\ref{GrafakosOh-Lemma}, with
	\[
	q_1=p,\qquad
	q_2=\infty,\qquad
	r_1=n,\qquad
	r_2=\frac{np}{n-p},
	\]
	it follows that
	\begin{align*}
		\|\PSI(\cdot,t)\|_{\dot{W}^{s,p}}
		\lesssim{}&
		\|\D \u(\cdot,t)\|_{\dot{W}^{s,p}}
		\|\u(\cdot,t)\|_{L^\infty}+
		\|\D \u(\cdot,t)\|_{L^{n}}
		\|\u(\cdot,t)\|_{
			\dot W^{s,\frac{np}{n-p}}}.
	\end{align*}
	
	By Proposition~\ref{Tecnical-Prop3}, the unit-modulus property
	\eqref{unit-modulus} gives
	\[
	\|\u(\cdot,t)\|_{L^\infty}=2^{\frac{n}{2}}.
	\]
	
	Moreover, \eqref{DuPsi}, the product norm inequality, and H\"older's
	inequality yield
	\[
	\|\D\u(\cdot,t)\|_{L^{n}}
	\leq
	\| \PSI(\cdot,t)\|_{L^{n}}
	\| \u(\cdot,t)\|_{L^\infty}
	=
	2^{\frac{n}{2}}
	\|\PSI(\cdot,t)\|_{L^{n}}.
	\]
	
	Consequently,
	\begin{align}
		\|\PSI(\cdot,t)\|_{\dot{W}^{s,p}}
		\lesssim
		2^{\frac{n}{2}}
		\left(
		\|\D\u(\cdot,t)\|_{\dot{W}^{s,p}}
		+
		\|\PSI(\cdot,t)\|_{L^{n}}
		\|\u(\cdot,t)\|_{\dot{W}^{s,\frac{np}{n-p}}}
		\right).
		\label{Psi-from-u}
	\end{align}
	
	By observing that
	\begin{eqnarray*}
\frac{1}{2}-\frac{n-2}{2n}=\frac{1}{n}	&\mbox{and}&
	\frac{1}{p}-\frac{n-p}{np}=\frac{1}{n},
	\end{eqnarray*}
	the Hardy--Littlewood--Sobolev inequality
	(cf.~\cite[Theorem 6.1.3]{GrafakosVol250}) yields
	\begin{eqnarray*}
		\|\PSI(\cdot,t)\|_{L^n}
		\lesssim
		\|\PSI(\cdot,t)\|_{\dot{H}^{\frac{n}{2}-1}}
		&\mbox{and}&
		\|\u(\cdot,t)\|_{
			\dot{W}^{s,\frac{np}{n-p}}}
		\lesssim
		\|\u(\cdot,t)\|_{\dot{W}^{s+1,p}}.
	\end{eqnarray*}
	
	Finally, the spectral formula
	\[
	\D=(-\Delta)^{\frac{1}{2}}\mathcal{H}
	\]
	for the Dirac operator \eqref{DiracOp} (see~\eqref{spectralDLaplace},~\eqref{DiracSymbol-Spherical} and \eqref{SpaceFractionalDirac}), together with the
boundedness of
	\(\mathcal{H}=\P_+-\P_-\) on 	\(L^p(\BR^n;\BC\otimes\cl_{0,n})\) (see~\eqref{ProjOp-RiezHilbert} and ~\eqref{ProjectionOp-Lp}), gives
	\[
	\|\D\u(\cdot,t)\|_{\dot W^{s,p}}
	\lesssim
	\|\u(\cdot,t)\|_{\dot W^{s+1,p}}.
	\]
	
	Substitution of the preceding estimates into \eqref{Psi-from-u}
	proves \eqref{Psi-from-u-Sobolev}.
\end{proof}

\subsection{Estimates involving the gauge nonlinearity \eqref{Nelson_pm}}\label{Estimates-Nonlinear-Subsection}

We now establish the following estimates for the gauge nonlinearity
\eqref{Nelson_pm}, which will be used below in the inhomogeneous
Strichartz estimates for \eqref{Schr-Alfa}.

\begin{proposition}
	\label{LeibnizSobolev-Theorem}
For \(n\geq3\), let $\PSI$ solve \eqref{BO-Dirac} and satisfy Assumptions~
\ref{AntiCommutation-Assumption}--\ref{Dyadic-Assumption}, and let \(\u\) be defined in
Proposition~\ref{Tecnical-Prop2}. 
	
	Suppose that
	\(
	\sigma\geq \frac{n}{2}-1
	\)
	is the regularity index in Assumption~\ref{Dyadic-Assumption}, and that
	\begin{equation}
		\label{M-Condition-sigma}
		M\leq
		\frac{1-2^{-(1+\varepsilon)}}{2C_{n,\sigma+1}},
	\end{equation}
	where \(C_{n,\sigma+1}>0\) denotes the constant associated with the
	Sobolev and fractional Leibniz estimates. Then, for every
	\(t\in[0,T]\),
	\begin{eqnarray}
		\label{GrafakosOh-Nu-embedding}
		\left\|
		\P_\pm{\bm N}\left(\PSI(\cdot,t),\u(\cdot,t)\right)
		\right\|_{\dot{H}^{\sigma+1}}
		\lesssim \nonumber \\ 
\lesssim		2^{\frac{n}{2}+1}
		\|\PSI(\cdot,t)\|_{\dot{H}^{\sigma+1}}
		\|\PSI(\cdot,t)\|_{L^\infty}
		\left(
		1+2\|\PSI(\cdot,t)\|_{\dot{H}^{\frac{n}{2}-1}}
		\right).
	\end{eqnarray}
\end{proposition}

\begin{proof}
	Fix \(t\in[0,T]\). By the definition of \({\bm N}\) in
	\eqref{Nelson_pm},
	\[
	\P_\pm{\bm N}\left(\PSI(\X,t),\u(\X,t)\right)
	=
	\P_\pm
	\left(
	\P_-\left[{\bm G}(\PSI(\X,t))\right]\u(\X,t)
	\right).
	\]
	
	Since \(\P_\pm\) are bounded on homogeneous Sobolev spaces by
	\eqref{ProjectionOp-Lp} and the
	observation following it, it follows that
	\begin{equation}
		\label{DPpm-Gu}
		\left\|
		\P_\pm{\bm N}\left(\PSI(\cdot,t),\u(\cdot,t)\right)
		\right\|_{\dot{H}^{\sigma+1}}
		\lesssim
		\left\|
		\P_-\left[{\bm G}(\PSI(\cdot,t))\right]\u(\cdot,t)
		\right\|_{\dot{H}^{\sigma+1}}.
	\end{equation}
	
	Next, we apply Lemma~\ref{GrafakosOh-Lemma} with
	\[
	p=q_1=2,
	\qquad
	q_2=\infty,
	\qquad
	r_1=n,
	\qquad
	r_2=\frac{2n}{n-2}.
	\]
	Using also the boundedness of \(\P_-\) in \(\dot{H}^{\sigma+1}\), we obtain
	\begin{eqnarray}
		\label{GrafakosOh-Nu}
		\left\|
		\P_\pm{\bm N}\left(\PSI(\cdot,t),\u(\cdot,t)\right)
		\right\|_{\dot{H}^{\sigma+1}}
		&\lesssim&
		\|{\bm G}(\PSI(\cdot,t))\|_{\dot{H}^{\sigma+1}}
		\|\u(\cdot,t)\|_{L^\infty}
		\nonumber\\
		&+&
		\|{\bm G}(\PSI(\cdot,t))\|_{L^{n}}
		\|\u(\cdot,t)\|_{\dot{W}^{\sigma+1,\frac{2n}{n-2}}}.
	\end{eqnarray}
	
	The homogeneous Sobolev embeddings yield
	\begin{align}
		\label{Sobolev-2}
		\|{\bm G}(\PSI(\cdot,t))\|_{L^{n}}
		&\lesssim
		\|{\bm G}(\PSI(\cdot,t))\|_{\dot{H}^{\frac{n}{2}-1}},
		\\
		\label{Sobolev-3}
		\|\u(\cdot,t)\|_{\dot{W}^{\sigma+1,\frac{2n}{n-2}}}
		&\lesssim
		\|\u(\cdot,t)\|_{\dot{H}^{\sigma+2}}.
	\end{align}
	
	The smallness condition \eqref{M-Condition-sigma} allows us to apply
	Proposition~\ref{Tecnical-Prop3} to \(s=\sigma+1\). Hence,
	\begin{eqnarray*}
		\|\u(\cdot,t)\|_{L^\infty}=2^{\frac{n}{2}}
		&\mbox{and}&
		\|\u(\cdot,t)\|_{\dot{H}^{\sigma+2}}
		\lesssim
		2^{\frac{n}{2}+1}
		\|\PSI(\cdot,t)\|_{\dot{H}^{\sigma+1}}.
	\end{eqnarray*}
	
	It remains to estimate the terms involving the nonlinear term \({\bm G}(\PSI(\cdot,t))\).
	Applying Lemma~\ref{GrafakosOh-Lemma} to \eqref{Gu}, with the same choice of exponents, gives
	\begin{eqnarray*}
		\|{\bm G}(\PSI(\cdot,t))\|_{\dot{H}^{\sigma+1}}
		&\lesssim&
		2
		\|\PSI(\cdot,t)\|_{\dot{H}^{\sigma+1}}
		\|\PSI(\cdot,t)\|_{L^\infty},
	\end{eqnarray*}
	and, by the homogeneous Sobolev embedding,
	\begin{eqnarray*}
		\|{\bm G}(\PSI(\cdot,t))\|_{L^{n}}
		&\lesssim&
		2
		\|\PSI(\cdot,t)\|_{\dot{H}^{\frac{n}{2}-1}}
		\|\PSI(\cdot,t)\|_{L^\infty}.
	\end{eqnarray*}
	
	By substituting the estimates above, along with \eqref{Sobolev-2} and \eqref{Sobolev-3}, into \eqref{GrafakosOh-Nu}, we obtain \eqref{GrafakosOh-Nu-embedding}.
\end{proof}

\section{Gauge correspondence and a priori estimates}\label{Gauge-Section}

\subsection{Proof of Theorem~\ref{Theorem-Gauge}}

\begin{proof}
	From Proposition \ref{Tecnical-Prop2}, the function $\u$, given by \eqref{upm-Gauge}, satisfies \eqref{unit-modulus} and \eqref{DuPsi}.
	Furthermore, using the spectral formula \eqref{SpaceFractionalDirac} for $\D$ and 
	$$
	\w(\X,t)=\D(-\Delta)^{-1}\PSI(\X,t),
	$$
	given by \eqref{w} (see Proposition \ref{Tecnical-Prop2}), we rewrite \eqref{BO-Dirac} as
	\begin{eqnarray}
		\label{FractionalDirac-w}
		i\partial_t \w(\X,t)-\Delta\mathcal{H}\w(\X,t)={\bm G}\left(\D\w(\X,t)\right)&,~(\X,t)\in\BR^{n}\times [0,T],
	\end{eqnarray}
	where $\mathcal{H}$ denotes the Riesz--Hilbert transform (see~\eqref{RieszHilbertOp}).
	
	Using \eqref{AntiCommutation-Condition} provided by Assumption \ref{AntiCommutation-Assumption}, we have
	\begin{eqnarray*}
		\D(\PSI(\X,t)\u(\X,t))
		&=&\big(\D \PSI(\X,t)\big)\u(\X,t)-(\PSI(\X,t))^\dagger \D \u(\X,t).
	\end{eqnarray*}
	
	Combining this with \eqref{Dirac-u}, given by Lemma \ref{Technical-Lemma0}, and \eqref{Dw} we obtain
	\begin{eqnarray*}
		-\Delta \u(\X,t)
		&=&\D\left(\D \u(\X,t)\right) \\
		&=&-i \D\left[(\D \w(\X,t))\,\u(\X,t)\right] \\
		&=&\left[-i(-\Delta)\w(\X,t)+\left(\PSI(\X,t)\right)^\dagger\PSI(\X,t)\right]\u(\X,t).
	\end{eqnarray*}
	
	Using the preceding identity together with \eqref{dt-u} of Lemma \ref{Technical-Lemma0}, we obtain
	\begin{align*}
		i\partial_t \u(\X,t)-\Delta \mathcal{H}\u(\X,t)
		&=-i\left[i\partial_t\w(\X,t)-\Delta \mathcal{H} \w(\X,t)-\mathcal{H}{\bm G}(\PSI(\X,t))\right]\u(\X,t).
	\end{align*}
	
Therefore, 
	\begin{eqnarray}
		\label{SchroedingerGauge}
		i\partial_t \u(\X,t)-\Delta \mathcal{H}\u(\X,t)
		={\bm N}(\PSI(\X,t),\u(\X,t)),~(\X,t)\in \BR^{n}\times [0,T],
	\end{eqnarray}
	where ${\bm N}(\cdot,\cdot)$ is given by \eqref{Nelson_pm}.
	
	By the properties of the Plemelj projectors $\P_\pm$, the components $\u_+=\P_+\u$ and $\u_-=\P_-\u$ satisfy \eqref{Schr-Alfa} if and only if $\w$ satisfies \eqref{FractionalDirac-w}. This establishes the equivalence between solutions $\PSI$ of \eqref{BO-Dirac} and the pair $(\u_+,\u_-)$ solving the coupled Schr\"odinger system \eqref{Schr-Alfa}.
\end{proof}

\subsection{Proof of Theorem~\ref{Theorem-Hs}}

\begin{proof}
	
	\medskip
	
	\underline{Proof of \eqref{Hs-Condition-Theorem} and~\eqref{Mixed-supHs}.}
	
	\medskip
	
	By Proposition~\ref{Tecnical-Prop1}, $\PSI(\cdot,t)\in H^{\sigma+1}(\BR^n;\BC\otimes\cl_{0,n})\cap L^\infty(\BR^n;\BC\otimes\cl_{0,n})$ for every
	$t\in[0,T]$. In particular,
	\[
	\|\PSI\|_{L^\infty_T \dot{H}^{\sigma+1}_{\X}}\leq \|\PSI\|_{L^\infty_T {H}^{\sigma+1}_{\X}} <\infty
	\qquad \mbox{and} \qquad
	\|\PSI\|_{L^\infty_T L^\infty_{\X}}<\infty.
	\]
	
	\medskip
	
	\underline{Proof of \eqref{Hs-Conservation}.}
	
	\medskip
	
	Let $\displaystyle \sigma \geq \frac{n}{2}-1$ and put
	\[
	\PSI_\sigma(\X,t)
	=
	(-\Delta)^{\frac{\sigma}{2}}\PSI(\X,t).
	\]
	
	By the reformulation of \eqref{Gu} in terms of the Clifford norm $|\cdot|_0$, given by \eqref{Gu0}, we obtain
	\begin{eqnarray}
		\label{Gu-s}
		(-\Delta)^{\frac{\sigma}{2}}\D{\bm G}(\PSI(\X,t))
		=
		-i2^{-n}\D(-\Delta)^{\frac{\sigma}{2}}
		\bigl(|\PSI(\X,t)|_0^2\bigr).
	\end{eqnarray}
	
	Applying \((-\Delta)^{\frac{\sigma}{2}}\) to
	\eqref{BO-Dirac} and multiplying on the left by
	\(-i2^n(\PSI_\sigma(\X,t))^\dagger\), we obtain
	\begin{eqnarray}
		\label{BO1}
		2^n(\PSI_\sigma(\X,t))^\dagger
		\Bigl(
		\partial_t\PSI_\sigma(\X,t)
		-i\D(-\Delta)^{\frac{1}{2}}\PSI_\sigma(\X,t)
		\Bigr)
		\nonumber\\
		=-(\PSI_\sigma(\X,t))^\dagger
		\D(-\Delta)^{\frac{\sigma}{2}}
		\bigl(|\PSI(\X,t)|_0^2\bigr).
	\end{eqnarray}
	
	Likewise, applying \((-\Delta)^{\frac{\sigma}{2}}\) to the
	\(\dagger\)-conjugate of \eqref{BO-Dirac} and multiplying on the
	right by \(i2^n\PSI_\sigma(\X,t)\), we obtain
	\begin{eqnarray}
		\label{BO2}
		2^n\Bigl(
		\partial_t(\PSI_\sigma(\X,t))^\dagger
		+i\D(-\Delta)^{\frac{1}{2}}\PSI_\sigma(\X,t)
		\Bigr)
		\PSI_\sigma(\X,t)
		\nonumber\\
		=
		\D(-\Delta)^{\frac{\sigma}{2}}
		\bigl(|\PSI(\X,t)|_0^2\bigr)
		\PSI_\sigma(\X,t).
	\end{eqnarray}
	
	Summing \eqref{BO1} and \eqref{BO2}, integrating over \(\BR^n\), and
	using \eqref{InnerP}, we obtain
	\begin{eqnarray*}
		\partial_t\langle \PSI_\sigma(\cdot,t),\PSI_\sigma(\cdot,t)\rangle
		&=&
		\int_{\BR^n}
		\left(
		\D(-\Delta)^{\frac{\sigma}{2}}
		\bigl(|\PSI(\X,t)|_0^2\bigr)\PSI_\sigma(\X,t)
		\right.
		\nonumber\\
		&&\left.
		-(\PSI_\sigma(\X,t))^\dagger
		\D(-\Delta)^{\frac{\sigma}{2}}
		\bigl(|\PSI(\X,t)|_0^2\bigr)
		\right)\dx.
	\end{eqnarray*}
	
	Applying \([\,\cdot\,]_0:\BC\otimes\cl_{0,n}\to\BC\), and using
	\eqref{L2-Norm}, \eqref{SobolevHomogeneous-Norm-Wsp}, and
	\eqref{AntiCommutation-Condition}, we obtain
	\begin{eqnarray}
		\label{BO-integral-s}
		\partial_t\|\PSI(\cdot,t)\|_{\dot{H}^\sigma}^2
		=
		2\int_{\BR^n}
		\left[
		\D(-\Delta)^{\frac{\sigma}{2}}
		\bigl(|\PSI(\X,t)|_0^2\bigr)
		\PSI_\sigma(\X,t)
		\right]_0\dx.
	\end{eqnarray}
	
	By \eqref{Sobolev-Condition} in Proposition~\ref{Tecnical-Prop1},
	\[
	\PSI(\cdot,t)\in
	\dot{H}^\sigma(\BR^n;\BC\otimes\cl_{0,n})
	\cap
	\dot{H}^{\sigma+1}(\BR^n;\BC\otimes\cl_{0,n}).
	\] 
	
	Hence, by
	Lemma~\ref{GrafakosOh-Lemma},
	\begin{eqnarray*}
		\left\|
		\D(-\Delta)^{\frac{\sigma}{2}}
		\bigl(|\PSI(\cdot,t)|_0^2\bigr)
		\right\|_{L^2}
		&=&
		\left\|
		|\PSI(\cdot,t)|_0^2
		\right\|_{\dot{H}^{\sigma+1}}
		\nonumber\\
		&\lesssim&
		\|\PSI(\cdot,t)\|_{\dot{H}^{\sigma+1}}
		\|\PSI(\cdot,t)\|_{L^\infty}.
	\end{eqnarray*}
	
	Thus,
	\[
	\D(-\Delta)^{\frac{\sigma}{2}}
	\bigl(|\PSI(\cdot,t)|_0^2\bigr)
	\PSI_\sigma(\cdot,t)
	\in
	L^1(\BR^n;\BC\otimes\cl_{0,n}),
	\]
	and the right-hand side of \eqref{BO-integral-s} is well-defined.
	
	By the product norm inequality and the Cauchy--Schwarz inequality,
	\begin{eqnarray*}
		\partial_t\|\PSI(\cdot,t)\|_{\dot{H}^\sigma}^2
		&\lesssim&
		2\left\|
		\D(-\Delta)^{\frac{\sigma}{2}}
		\bigl(|\PSI(\cdot,t)|_0^2\bigr)
		\right\|_{L^2}
		\|\PSI_\sigma(\cdot,t)\|_{L^2}
		\nonumber\\
		&\lesssim&
		2\|\PSI(\cdot,t)\|_{\dot{H}^{\sigma+1}}
		\|\PSI(\cdot,t)\|_{L^\infty}
		\|\PSI(\cdot,t)\|_{\dot{H}^\sigma}
		\nonumber\\
		&\leq&
		2\|\PSI\|_{L^\infty_T \dot{H}^{\sigma+1}_{\X}}
		\|\PSI\|_{L^\infty_T L^\infty_{\X}}
		\|\PSI(\cdot,t)\|_{\dot{H}^\sigma}.
	\end{eqnarray*}
	
	Moreover, application of Young's inequality yields
	\begin{eqnarray}
	\label{BO-integral-ineq}
	\partial_t\|\PSI(\cdot,t)\|_{\dot{H}^\sigma}^2
	\lesssim
	\|\PSI\|_{L^\infty_T H^{\sigma+1}_{\X}}^2
	\|\PSI\|_{L^\infty_T L^\infty_{\X}}^2
	+
	\|\PSI(\cdot,t)\|_{\dot{H}^\sigma}^2.
\end{eqnarray}

We now use assumption \eqref{continuityPSI} to prove \eqref{Hs-Conservation}. The continuity in time of the nonlinear term follows from Lemma~\ref{GrafakosOh-Lemma}. More precisely, 
	\begin{eqnarray*}
		\label{continuityPSI2}
		t\mapsto
		\D(-\Delta)^{\frac{\sigma}{2}}
		\bigl(|\PSI(\cdot,t)|_0^2\bigr)
	\end{eqnarray*}
	belongs to
	\(
	C\bigl([0,T];L^2(\BR^n;\BC\otimes\cl_{0,n})\bigr).
	\)
	
	Indeed, for every $\X\in \BR^n$ and $t,t_0\in[0,T]$,
	\[
	|\PSI(\X,t)|_0^2-|\PSI(\X,t_0)|_0^2
	=
	2^n\left[
	(\PSI(\X,t))^\dagger\PSI(\X,t)
	-
	(\PSI(\X,t_0))^\dagger\PSI(\X,t_0)
	\right]_0,
	\]
	where
	\begin{eqnarray*}
		(\PSI(\X,t))^\dagger\PSI(\X,t)
		-
		(\PSI(\X,t_0))^\dagger\PSI(\X,t_0)
		&=&
		\left(
		\PSI(\X,t)-\PSI(\X,t_0)
		\right)^\dagger
		\PSI(\X,t)
		\nonumber\\
		&+&
		(\PSI(\X,t_0))^\dagger
		\left(
		\PSI(\X,t)-\PSI(\X,t_0)
		\right).
	\end{eqnarray*}
	
	Thus, Lemma~\ref{GrafakosOh-Lemma} and
	\eqref{normIneq-Wsp} imply
	\begin{align*}
		&\left\|~
		|\PSI(\cdot,t)|_0^2-
		|\PSI(\cdot,t_0)|_0^2~
		\right\|_{\dot{H}^{\sigma+1}}
		\\
		&\qquad\lesssim
		\|\PSI(\cdot,t)-\PSI(\cdot,t_0)\|_{H^{\sigma+1}}
		\left(
		\|\PSI(\cdot,t)\|_{L^\infty}
		+
		\|\PSI(\cdot,t_0)\|_{L^\infty}
		\right)
		\\
		&\qquad\quad+
		\|\PSI(\cdot,t)-\PSI(\cdot,t_0)\|_{L^\infty}
		\left(
		\|\PSI(\cdot,t)\|_{\dot{H}^{\sigma+1}}
		+
		\|\PSI(\cdot,t_0)\|_{\dot{H}^{\sigma+1}}
		\right).
	\end{align*}
	
	The right-hand side converges to zero as \(t\to t_0\). This proves the
	asserted continuity.
	
	Consequently, the right-hand side of \eqref{BO-integral-s} is continuous on
	\([0,T]\), and
	\[
\PSI \in C^1\bigl([0,T];\dot{H}^\sigma(\BR^n;\BC\otimes\cl_{0,n})\bigr).
	\]

 Let us denote by $A$ the constant
	\[
	A:=
	\|\PSI\|_{L^\infty_T \dot{H}^{\sigma+1}_{\X}}^2
	\|\PSI\|_{L^\infty_T L^\infty_{\X}}^2.
	\]

	Multiplication of both sides of \eqref{BO-integral-ineq} by \(e^{-t}\) gives
	\[
	\partial_t
	\left(
	e^{-t}\|\PSI(\cdot,t)\|_{\dot{H}^\sigma}^2
	\right)
\lesssim
e^{-t}A.
	\]
	
Integration over \([0,t]\), yields
		\[
	\|\PSI(\cdot,t)\|_{\dot{H}^\sigma}^2
\lesssim
	e^t\|\PSI(\cdot,0)\|_{\dot{H}^\sigma}^2
	+
	A(e^t-1).
	\]
	
	Consequently, for every \(t\in[0,T]\),
	\[
	\|\PSI(\cdot,t)\|_{\dot{H}^\sigma}^2
	\lesssim
	e^T\left(
	\|\PSI(\cdot,0)\|_{\dot{H}^\sigma}^2
	+
	T\|\PSI\|_{L^\infty_T \dot{H}^{\sigma+1}_{\X}}^2
	\|\PSI\|_{L^\infty_T L^\infty_{\X}}^2
	\right).
	\]
	
	Finally, the inequality
	\[
	a+b\leq\left(\sqrt a+\sqrt b\right)^2,
	\qquad
	a,b\geq0,
	\]
	gives
	\[
	\|\PSI(\cdot,t)\|_{\dot{H}^\sigma}
	\lesssim
	e^{\frac{T}{2}}
	\left(
	\|\PSI(\cdot,0)\|_{\dot{H}^\sigma}
	+
	\sqrt{T}
	\|\PSI\|_{L^\infty_T \dot{H}^{\sigma+1}_{\X}}
	\|\PSI\|_{L^\infty_T L^\infty_{\X}}
	\right),
	\]
	for all \(t\in[0,T]\). This proves \eqref{Hs-Conservation}.
\end{proof}

\begin{remark}[Lack of conservation laws]
	The integrand in \eqref{BO-integral-s} has no definite sign.
	Consequently, the anti-commutation condition
	\eqref{AntiCommutation-Condition} from Assumption
	\ref{AntiCommutation-Assumption} does not imply that the right-hand side
	of the energy identity vanishes, so that
	$\partial_t\|\PSI(\cdot,t)\|_{\dot{H}^\sigma}^2$ is not identically zero.
	The argument used in the proof of Theorem~\ref{Theorem-Hs} therefore
	yields an a priori estimate via the Cauchy--Schwarz inequality and
	Gronwall's lemma, rather than an exact conservation law for the
	$\dot{H}^\sigma-$seminorm. Whether the scalar contribution in \eqref{BO-integral-s} vanishes under
	a suitable invariant condition on the solution class remains open for
	$n\geq 3$.
\end{remark}

\section{From Strichartz estimates to a priori bounds}\label{Strichartz-Section}


By Theorem~\ref{Theorem-Gauge}, the Plemelj components \(\u_\pm\) satisfy the coupled Schr\"odinger system \eqref{Schr-Alfa}. We write this system in Duhamel form and apply the homogeneous and inhomogeneous Strichartz estimates of Ginibre and Velo \cite{GinibreVelo92}, together with the endpoint estimate of Keel and Tao \cite{KeelTao98} (see also \cite[Subsection~2.3]{Tao06}). The resulting bounds for \(\u_\pm\), combined with the estimate for the gauge nonlinearity \eqref{Nelson_pm} and the comparison estimate relating \(\u\) to \(\PSI\), give the claimed a priori bound for \(\PSI\).

\subsection{Homogeneous and inhomogeneous estimates}

By Duhamel's formula, the solutions $(\u_+,\u_-)$ of \eqref{Schr-Alfa} with initial data $(\u_+(\cdot,0),\u_-(\cdot,0))$ satisfy
\begin{eqnarray}
	\label{Duhamel-upm}\u_\pm(\X,t)=e^{\mp it\Delta}\u_\pm(\X,0) \pm i \int_0^t e^{\mp i(t-\tau)\Delta} \P_\pm{\bm N}\left(\PSI(\X,\tau),\u(\X,\tau)\right) d\tau,
\end{eqnarray}
where ${\bm N}(\cdot,\cdot)$ denotes the gauge nonlinearity defined in \eqref{Nelson_pm}.

The two terms in \eqref{Duhamel-upm} are controlled by the homogeneous and inhomogeneous mapping properties of
\begin{eqnarray*}
	\f \mapsto e^{\mp it\Delta}\f &\mbox{and}& {\bm F}(\cdot,\cdot) \mapsto \int_0^t e^{\mp i(t-\tau)\Delta} {\bm F}(\cdot,\tau) d\tau,
\end{eqnarray*}
respectively.

\begin{theorem}[see \cite{GinibreVelo92,KeelTao98}]\label{StrichartzTheorem}
Let $(q,p)$ and $(\widetilde{q},\widetilde{p})$ be sharp Schr\"odinger admissible pairs in the sense of Definition \ref{SchroedingerPair}. Then the homogeneous estimate
	\begin{eqnarray}
		\left\| e^{\mp i t\Delta}\f \right\|_{L^q_TL_{\X}^p}\lesssim \| \f\|_{L^2}
	\end{eqnarray}
	and the inhomogeneous estimate
	\begin{eqnarray}
		\label{inhomogeneousStrichartz}
		\left\| \int_0^t e^{\mp i(t-\tau)\Delta} {\bm F}(\cdot,\tau) d\tau\right\|_{L^q_TL_{\X}^p}\lesssim~\|{\bm F}\|_{L^{\widetilde{q}'}_TL^{\widetilde{p}'}_{\X}}
	\end{eqnarray}
	hold, where
	\begin{eqnarray*}
		\frac{1}{\widetilde{q}}+\frac{1}{\widetilde{q}'}=\frac{1}{\widetilde{p}}+\frac{1}{\widetilde{p}'}=1.
	\end{eqnarray*}
\end{theorem}

We shall apply Theorem~\ref{StrichartzTheorem} to \eqref{Duhamel-upm}, taking
\begin{eqnarray}
	\label{Strichart-f-F}
	\f=\u_\pm(\cdot,0)
	\qquad\mbox{and}\qquad
	{\bm F}=\P_\pm{\bm N}(\PSI,\u).
\end{eqnarray}

This reduces the proof of the a priori estimate for \(\PSI\) to the
control of the initial gauge datum $\displaystyle \u(\cdot,0)=\u_+(\cdot,0)+\u_-(\cdot,0)$ and of the gauge nonlinearity \eqref{Nelson_pm}.

\subsection{Proof of Theorem \ref{StrichartzTheorem-PSI}}

\begin{proof}
	
	Fix \(t\in[0,T]\). By Proposition~\ref{Tecnical-Prop3} and
	Theorem~\ref{Theorem-Hs},
	\begin{equation}
		\label{Psi-from-u-Hs-Conservation}
		\|\PSI(\cdot,t)\|_{\dot W^{s,p}}
		\lesssim
		2^{\frac n2}
		\|\u(\cdot,t)\|_{\dot W^{s+1,p}}
		\left(
		1+\mathcal E_{\frac n2-1,T}(\PSI)
		\right).
	\end{equation}
	Taking \(s=\sigma\) and the \(L^q(0,T)-\)norm yields
	\begin{equation}
		\label{Psi-from-u-Hs-Mixed}
		\|\PSI\|_{L^q_T\dot W^{\sigma,p}_{\X}}
		\lesssim
		2^{\frac n2}
		\|\u\|_{L^q_T\dot W^{\sigma+1,p}_{\X}}
		\left(
		1+\mathcal E_{\frac n2-1,T}(\PSI)
		\right).
	\end{equation}
	
 Applying
	Theorem~\ref{StrichartzTheorem} to \eqref{Duhamel-upm}, with
\((q,p)\) and \((\widetilde q,\widetilde p)=(\infty,2)\), which are Schr\"odinger admissible for \(n\geq 3\), gives
	\begin{align}
		\|\u_\pm\|_{L^q_T\dot W^{\sigma+1,p}_{\X}}
		&\lesssim
		\|\u_\pm(\cdot,0)\|_{\dot H^{\sigma+1}}
		+
		\|\P_\pm{\bm N}(\PSI,\u)\|_{L^1_T\dot H^{\sigma+1}_{\X}}
		\nonumber\\
		&\lesssim
		2^{\frac n2}\|\PSI(\cdot,0)\|_{\dot H^\sigma}
		+
		\|\P_\pm{\bm N}(\PSI,\u)\|_{L^1_T\dot H^{\sigma+1}_{\X}}.
		\label{Strichartz-upm-Hsigma}
	\end{align}
	
In the last estimate, we used \eqref{Hs-u} and the commutation between
	\((-\Delta)^{\frac{\sigma+1}{2}}\), \(e^{\mp it\Delta}\), and the Plemelj projectors
	\(\P_\pm\).
	
	For the nonlinear terms of \eqref{Schr-Alfa}, Proposition~\ref{LeibnizSobolev-Theorem} gives
	\begin{align*}
		\|\P_\pm{\bm N}(\PSI(\cdot,t),\u(\cdot,t))\|_{\dot H^{\sigma+1}} 
		\lesssim &
		2^{\frac n2+1}
		\|\PSI(\cdot,t)\|_{\dot H^{\sigma+1}}
		\|\PSI(\cdot,t)\|_{L^\infty}
		\left(
		1+2\|\PSI(\cdot,t)\|_{\dot H^{\frac n2-1}}
		\right)
		\nonumber\\
		\lesssim &
		2^{\frac n2+1}
		\|\PSI\|_{L^\infty_T\dot H^{\sigma+1}_{\X}}
		\|\PSI\|_{L^\infty_TL^\infty_{\X}}
		\left(
		1+2\mathcal E_{\frac n2-1,T}(\PSI)
		\right),
	\end{align*}
	where the second inequality follows from Theorem~\ref{Theorem-Hs}, with
	\(\displaystyle s=\frac{n}{2}-1\). Therefore,
	\begin{align}
		\|\P_\pm{\bm N}(\PSI,\u)\|_{L^1_T\dot H^{\sigma+1}_{\X}}
		\lesssim{}&
		2^{\frac n2+1}T
		\|\PSI\|_{L^\infty_T\dot H^{\sigma+1}_{\X}}
		\|\PSI\|_{L^\infty_TL^\infty_{\X}}
		\left(
		1+2\mathcal E_{\frac{n}{2}-1,T}(\PSI)
		\right).
		\label{Nonlinear-u-Hsigma}
	\end{align}
	
	Since \(\u=\u_++\u_-\), \eqref{Strichartz-upm-Hsigma} and
	\eqref{Nonlinear-u-Hsigma} imply
	\begin{align}
		\|\u\|_{L^q_T\dot W^{\sigma+1,p}_{\X}}
		\lesssim{}&
		2^{\frac n2+1}
		\Bigg[
		\|\PSI(\cdot,0)\|_{\dot H^\sigma}
		+
		2T
		\|\PSI\|_{L^\infty_T\dot H^{\sigma+1}_{\X}}
		\|\PSI\|_{L^\infty_TL^\infty_{\X}}
		\nonumber\\
		&\hspace{2.5cm}\times
		\left(
		1+2\mathcal E_{\frac n2-1,T}(\PSI)
		\right)
		\Bigg].
		\label{u-Strichartz-preE}
	\end{align}
	By the definition of \(\mathcal E_{\sigma,T}(\PSI)\),
	\[
	\|\PSI(\cdot,0)\|_{\dot H^\sigma}
	\leq
	e^{-\frac T2}\mathcal E_{\sigma,T}(\PSI)
	\]
	and
	\[
	T
	\|\PSI\|_{L^\infty_T\dot H^{\sigma+1}_{\X}}
	\|\PSI\|_{L^\infty_TL^\infty_{\X}}
	\leq
	e^{-\frac T2}\sqrt T\,
	\mathcal E_{\sigma,T}(\PSI).
	\]
	Hence,
	\begin{align}
		\|\u\|_{L^q_T\dot W^{\sigma+1,p}_{\X}}
		&\lesssim
		2^{\frac n2+1}e^{-\frac T2}
		\mathcal E_{\sigma,T}(\PSI)
		\left[
		1+2\sqrt T
		\left(
		1+2\mathcal E_{\frac n2-1,T}(\PSI)
		\right)
		\right]
		\nonumber\\
		&\lesssim
		2^{\frac n2+1}	\left(1+2\sqrt T\right)e^{-\frac{T}{2}}
		\mathcal E_{\sigma,T}(\PSI)
		\left(
		1+2\mathcal E_{\frac n2-1,T}(\PSI)
		\right).
		\label{Strichartz-uE-Hsigma}
	\end{align}
	
	Substituting \eqref{Strichartz-uE-Hsigma} into
	\eqref{Psi-from-u-Hs-Mixed} proves
	\eqref{StrichartzEstimate-Hgamma}.
\end{proof}

\begin{remark}[Choice of the Strichartz space]
For any sharp admissible pair \((q,p)\), we first estimate the gauge components \(\u_\pm\) in \(L^q_T\dot{W}^{\sigma+1,p}_x\), using the pair \((\infty,2)\) for the inhomogeneous term. The latter is admissible because the gauge nonlinearity belongs to \(L^1_T\dot{H}^{\sigma+1}_{\X}\). We observe also that the \(L^\infty_{T}L^\infty_{\X}-\)term appearing in \eqref{Nonlinear-u-Hsigma} is not itself a consequence of the critical Sobolev regularity $\displaystyle s=\frac{n}{2}-1$ (see Proposition~\ref{Tecnical-Prop1}). Finally, Proposition~\ref{Tecnical-Prop3} transfers the estimate for \(\u\) to \(\PSI\), giving the bound in \(L^q_T\dot{W}^{\sigma,p}_x\).
\end{remark}

\section*{AI disclosure}

AI tools were used to proofread certain computations and enhance the manuscript's readability. After using these tools, the author reviewed and edited the content as needed and takes full responsibility for the content of the manuscript.

\end{document}